\documentclass[11pt]{article}

\usepackage[margin=1in]{geometry}
\usepackage{amsmath,amssymb,amsthm,mathtools}
\usepackage{enumitem}
\usepackage{booktabs,tabularx,array}
\usepackage{microtype}
\usepackage{xcolor}
\usepackage[
  colorlinks=true,
  linkcolor=blue,
  citecolor=blue,
  urlcolor=blue,
  filecolor=blue,
  linktoc=all,
  bookmarks=true,
  bookmarksopen=true,
  bookmarksnumbered=true,
  pdfpagemode=UseOutlines
]{hyperref}
\hypersetup{
  pdftitle={Clumsy and Careless: Stationary-Entry Flux in Non-monotone Coupon Collectors},
  pdfauthor={Christopher D. Long},
  pdfsubject={Stationary-entry flux for non-monotone coupon collectors},
  pdfkeywords={coupon collector, terminal-condition flux calculus, stationary-entry flux, Markov chains}
}
\usepackage{bookmark}

\newtheorem{theorem}{Theorem}[section]
\newtheorem{lemma}[theorem]{Lemma}
\newtheorem{proposition}[theorem]{Proposition}
\newtheorem{corollary}[theorem]{Corollary}
\theoremstyle{definition}
\newtheorem{definition}[theorem]{Definition}
\newtheorem{remark}[theorem]{Remark}

\newcommand{\Pbb}{\mathbb P}
\newcommand{\Ebb}{\mathbb E}
\newcommand{\Nbb}{\mathbb N}
\newcommand{\Rbb}{\mathbb R}
\newcommand{\1}{\mathbf 1}
\newcommand{\TV}{\mathrm{TV}}
\newcommand{\Exp}{\mathrm{Exp}}
\newcommand{\Bin}{\mathrm{Bin}}
\newcommand{\Var}{\mathrm{Var}}

\newcommand{\OO}{\mathrm{O}}

\newcommand{\floor}[1]{\left\lfloor #1\right\rfloor}
\newcommand{\ceil}[1]{\left\lceil #1\right\rceil}

\newcommand{\abs}[1]{\left|#1\right|}

\title{Clumsy and Careless:\\
Stationary-Entry Flux in Non-monotone Coupon Collectors}
\author{Christopher D. Long}
\date{}

\begin{document}
\maketitle

\begin{abstract}
We study non-monotone coupon collectors in which collected coupons may be lost
or reset before completion.  The organizing idea is to identify how the process
first reaches completion.  In monotone collectors this is usually controlled by
a terminal cloud: the residual set or count of terminal conditions still
unsatisfied.  In the recurrent loss models considered here, completion is
instead governed by rare new entries into the all-present state, except in the
reset-button model, where exact resets give an independent-excursion structure.

We first prove a finite stationary-entry theorem: stationary-entry flux, fast
mixing, and one-block clump control imply an exponential hitting-time law.  For
comparison, we revisit the reset-button collector of Jockovi\'c and Todi\'c.
Exact regeneration at the empty state gives the identity
\(s_n=\Ebb q_n^{C_n}\), recovers the known beta-function expectation, and
yields rare-success exponential limits.

The stationary-entry theorem is then applied to the fixed-loss clumsy and
careless coupon collectors.  For the clumsy collector, with selected-type loss
probability \(p\) and \(q=1-p\), the stationary-entry flux is \(p q^n\), and
\[
        p q^n T_n \Rightarrow \Exp(1).
\]
For the careless collector, with the same fixed loss probability \(p\) and
\(q=1-p\), under the post-loss convention of Cruciani and Dudeja, the
completion scale is governed by the stationary high tail of the count chain.
Writing \(\mu_n\) for the stationary-entry flux, we prove
\[
        \mu_n\sim \frac{1}{(q;q)_\infty}
        \frac{n!}{n^n}q^{n(n+1)/2},
        \qquad
        \mu_n T_{n,p}\Rightarrow \Exp(1).
\]
Thus
\[
        \log \Ebb T_{n,p}
        =\frac{n(n+1)}{2}|\log q|+n+\OO(\log n).
\]
The careless scale is therefore not the one-point marginal scale suggested by
the natural independence heuristic.  The extra factor comes from an ordered lucky climb:
a run through the upper tail in which each step gains one coupon and all coupons
currently present survive the loss step.  We also analyze a combined
clumsy-careless model, showing that the same high-tail entry
mechanism is stable under selected-type refresh followed by global thinning.
\end{abstract}

\setcounter{tocdepth}{2}
\tableofcontents

\section{Introduction}

Classical coupon collector problems are monotone: once a coupon type has been collected, that part of the task remains complete; see, for example, \cite{ErdosRenyi,Feller,Holst}.  A standard proof strategy studies a terminal cloud \(W_n(t)\), by which we mean a residual set, count, or field of terminal conditions that remain unsatisfied at time \(t\).  Completion is then the event
\[
        \{T_n\le t\}=\{W_n(t)=0\},
\]
and a Poisson limit for \(W_n(t)\) gives the usual Gumbel-type laws.

This paper studies coupon collectors where collected coupons may be lost.  The target set \(A_n\), usually the all-present state, is no longer absorbing.  The first step is therefore to identify how the process reaches completion.  In an exactly reset model, completion is a successful reset-free excursion.  In the clumsy and careless loss models, completion occurs when the chain newly enters the all-present state from outside.  Related entry formulations appear in the expiring coupon collector, where completion is governed by rare entries of a sliding window into the surjective state \cite{LongExpiring}.

A recurring point is that the reciprocal stationary mass of \(A_n\) is not necessarily the right time scale: visits to \(A_n\) may arrive in clumps.  The scale is governed not by how often the chain is found in \(A_n\), but by how often it newly crosses into that state from outside.  We call this rate the stationary-entry flux.  In the clumsy model the entry rate is \(p\) times the stationary mass \(q^n\), so the time scale is \(1/(p q^n)\), rather than \(q^{-n}\).  In the careless model the effect is much larger, because the stationary high tail is not governed by the independent one-point product heuristic.  Completion is governed by an ordered lucky climb: near the top of the count chain, the process must repeatedly draw a missing coupon while all currently held coupons survive the loss step.  The associated product is
\[
        \prod_{k=0}^{n-1}\frac{n-k}{n}q^{k+1}
        =\frac{n!}{n^n}q^{n(n+1)/2}.
\]

The stationary-entry principle used below belongs to a classical line of work on exponential hitting laws and metastability for Markov chains.  Aldous proved that rapid convergence to stationarity gives almost exponential hitting times for rare sets \cite{Aldous1982}, with refinements for rare events in reversible chains by Aldous and Brown \cite{AldousBrown1992}.  Broader metastability approaches include the pathwise viewpoint of Cassandro, Galves, Olivieri, and Vares \cite{CassandroGalvesOlivieriVares1984}, the large-deviation treatment of Olivieri and Vares \cite{OlivieriVares2005}, and the potential-theoretic approach of Bovier and den Hollander \cite{BovierDenHollander2015}.  Our contribution is a finite, block-checkable stationary-entry criterion tailored to coupon-loss processes, together with sharp entry-rate computations for the clumsy, careless, and combined collectors.

\subsection{Main results and mechanisms}

The symbols used for survival or refresh probabilities are local to the corresponding model sections: \(q_n\) is used in the reset-button model, \(q=1-p\) in the clumsy and careless fixed-loss models, and \(Q=1-\alpha\), \(S=1-\beta\) in the combined model.  For \(a\in(0,1)\), write
\[
        (a;a)_\infty:=\prod_{r=1}^{\infty}(1-a^r).
\]
The main fixed-parameter stationary-entry laws are the following, where each \(\mu_n\) denotes the corresponding stationary-entry flux:
\begin{align*}
        p q^n T_n^{\mathrm{clumsy}} &\Rightarrow \Exp(1),\\
        \mu_n^{\mathrm{careless}}T_{n,p}^{\mathrm{careless}} &\Rightarrow \Exp(1),
        \qquad
        \mu_n^{\mathrm{careless}}
        \sim
        \frac{1}{(q;q)_\infty}\frac{n!}{n^n}q^{n(n+1)/2},\\
        \mu_n^{\alpha,\beta}T_n^{\alpha,\beta} &\Rightarrow \Exp(1),
        \qquad
        \mu_n^{\alpha,\beta}
        \sim
        \frac{1}{(S;S)_\infty}\frac{n!}{n^n}Q^nS^{n(n+1)/2}.
\end{align*}
The proofs separate according to how completion is reached:
\begin{center}
\footnotesize
\begin{tabularx}{\textwidth}{@{}>{\raggedright\arraybackslash}p{0.19\textwidth}>{\raggedright\arraybackslash}p{0.22\textwidth}>{\raggedright\arraybackslash}p{0.25\textwidth}>{\raggedright\arraybackslash}X@{}}
\toprule
Model & Target set & Mechanism & Scale or method \\
\midrule
Classical monotone collector & all coupons present & terminal cloud & Poisson terminal-cloud limit \\
Reset-button collector & all standard coupons present before the next reset & reset-free successful excursion & exact regeneration through \(s_n=\Ebb q_n^{C_n}\) \\
Clumsy collector & all coupons present & stationary-entry flux & \(\mu_n=pq^n\) \\
Careless collector & all coupons present post-loss & high-tail stationary-entry flux & \(\mu_n\sim(q;q)_\infty^{-1}(n!/n^n)q^{n(n+1)/2}\) \\
Combined collector & all coupons present post-thinning & high-tail stationary-entry flux & \(\mu_n\sim(S;S)_\infty^{-1}(n!/n^n)Q^nS^{n(n+1)/2}\) \\
\bottomrule
\end{tabularx}
\end{center}

The first contribution is the stationary-entry module
\[
\boxed{
\text{stationary-entry flux}+\text{fast mixing}+\text{one-block clump control}
\Longrightarrow \text{exponential hitting law}.}
\]
The theorem is stated in finite block form, so the hypotheses can be checked directly in coupon-loss chains without reversibility.

The reset-button collector of Jockovi\'c and Todi\'c \cite{JockovicTodic2024} is structurally different from the stationary-entry models.  Their paper introduces the model, gives finite block formulas for the unequal-probability distribution, and derives the beta-function expectation in the equal-probability continuing-after-reset case.  Here we isolate the exact regeneration mechanism: each reset returns the process to the empty state, and a reset-free excursion succeeds with probability
\[
        s_n=\Ebb q_n^{C_n},
\]
where \(C_n\) is the ordinary coupon collector time measured in standard-coupon draws.  This gives a compact probability generating function and rare-success exponential limits.  Additional reset asymptotic regimes are recorded in Appendix~\ref{app:reset-asymptotic-regimes}.

For the clumsy collector of Aggarwal and Garoni \cite{AggarwalGaroni}, each step selects one type and refreshes it to present with probability \(q=1-p\) and absent with probability \(p\).  The all-present state has stationary mass \(q^n\), but visits occur in clumps, and the stationary-entry flux is \(pq^n\).  Hence
\[
        p(1-p)^nT_n\Rightarrow\Exp(1),
\]
with uniform exponential tails and convergence of all fixed moments.  Equivalently, after centering by its mean and scaling by its standard deviation, \(T_n\) converges to \(\Exp(1)-1\), not to a Gumbel law.  A contemporaneous and independent manuscript of Attrill and Garoni, submitted to arXiv concurrently with this work, treats the clumsy collector by language-generating-function methods; its fixed-\(p\) results overlap the clumsy result above, and it also analyzes moving-\(p\) regimes \cite{AttrillGaroni2026}.

For the careless collector of Cruciani and Dudeja \cite{CrucianiDudeja}, we use their post-loss convention: one coupon is drawn, then every currently held coupon is independently lost with probability \(p\), and completion is checked after that loss step.  Their mean-field and lower-bound scale is the product-marginal scale \(q_*^{-n}\), where
\[
        q_*:=\frac{1-p}{1-p+np}.
\]
The fixed-\(p\) stationary-entry scale is instead
\[
        \Ebb T_{n,p}
        \sim (q;q)_\infty\frac{n^n}{n!}q^{-n(n+1)/2}.
\]
Thus the logarithmic scale changes from \(n\log n+O(n)\) under the marginal heuristic to
\[
        \frac12|\log q|n^2+O(n).
\]
The extra factor is not a constant-order correction; it is the ordered lucky-climb cost through the upper tail of the count chain.

Finally, the combined clumsy-careless model verifies the stability of the high-tail entry mechanism.  A selected coupon is refreshed successfully with probability \(Q\), and then all present coupons survive a global thinning with probability \(S\).  The same high-tail entry mechanism gives
\[
        \mu_n^{\alpha,\beta}
        \sim (S;S)_\infty^{-1}\frac{n!}{n^n}Q^nS^{n(n+1)/2},
\]
followed by the same exponential hitting and moment limits.  The repetitive verification details are deferred to Appendix~\ref{app:combined-high-tail}.

\section{A stationary-entry flux module}\label{sec:stationary-entry-module}

Let $X_n=(X_n(t))_{t\ge0}$ be a finite irreducible Markov chain, observed in discrete time, with transition matrix $P_n$ and stationary law $\pi_n$.  Let $A_n$ be a rare target set.  Define
\[
        T_{A_n}:=\inf\{t\ge0:X_n(t)\in A_n\}.
\]
For $t\in\mathbb Z_{\ge1}$, let
\[
        E_t:=\1\{X_n(t-1)\notin A_n,\ X_n(t)\in A_n\}
\]
be the indicator of a new entry into the target set.  The stationary-entry flux is
\[
        \mu_n:=\Pbb_{\pi_n}(E_1=1).
\]

\begin{definition}[Mixing coefficient]
For $h\ge0$, define
\[
        \alpha_n(h):=\sup_x\|P_n^h(x,\cdot)-\pi_n\|_{\TV}.
\]
\end{definition}

\begin{lemma}[Path-event total-variation comparison]\label{lem:path-tv-comparison}
Let $P$ be a Markov transition matrix on a finite state space, and let $\eta$ and $\zeta$ be two initial laws.  For any integer $b\ge0$, the laws on paths $(X(0),\ldots,X(b))$ generated from $\eta$ and $\zeta$ differ in total variation by at most $\|\eta-\zeta\|_{\TV}$.  Consequently, for every event $B$ measurable with respect to $(X(0),\ldots,X(b))$,
\[
        \left|\Pbb_\eta(B)-\Pbb_\zeta(B)\right|
        \le \|\eta-\zeta\|_{\TV}.
\]
In particular, if $B$ is the event that a block of length $b$ contains at least one new entry into $A_n$, then its probability from an initial law within $\varepsilon$ of stationarity differs from its stationary probability by at most $\varepsilon$.
\end{lemma}

\begin{proof}
The path law from an initial distribution $\eta$ is obtained from $\eta$ by applying the Markov kernel
\[
        x_0\mapsto \Pbb_{x_0}\bigl((X(0),\ldots,X(b))\in\cdot\bigr).
\]
Markov kernels contract total variation distance.  Hence the two path laws are within $\|\eta-\zeta\|_{\TV}$, and the event bound follows from the definition of total variation distance.
\end{proof}

The following theorem is stated with block inputs.  It is a coupon-loss version of the classical almost-exponential hitting principle \cite{Aldous1982}, with the entry flux replacing stationary mass in order to account for clumping.  We use total-variation mixing and coupling estimates in the standard form; see, for example, \cite{LevinPeres}.  The block formulation makes the module directly checkable in concrete coupon-loss models.

\begin{theorem}[Stationary-entry hitting]\label{thm:stationary-entry}
Let $x_n$ be a sequence of initial states.  Assume
\[
        \pi_n(A_n)\to0,
        \qquad
        \mu_n\to0.
\]
Suppose there are integers $b_n,h_n,r_n\to\infty$ such that
\[
        h_n=o(b_n),\qquad b_n\mu_n\to0,
        \qquad \alpha_n(h_n)=o(b_n\mu_n),
\]
and such that the following two assumptions hold.
\begin{enumerate}[label=\textnormal{(M\arabic*)},leftmargin=2.4em]
\item \textbf{Stationary one-block law.}  With
\[
        N_n^{(b)}:=\sum_{t=1}^{b_n}E_t,
\]
one has
\[
        \Pbb_{\pi_n}(N_n^{(b)}\ge1)=b_n\mu_n(1+o(1)).
\]
\item \textbf{Negligible initial burn-in.}  For this sequence $x_n$,
\[
        \Pbb_{x_n}(T_{A_n}\le r_n)=o(1),
        \qquad
        \|P_n^{r_n}(x_n,\cdot)-\pi_n\|_{\TV}=o(1),
        \qquad
        r_n\mu_n=o(1).
\]
\end{enumerate}
Then
\[
        \mu_n T_{A_n}\Rightarrow \Exp(1)
\]
under $\Pbb_{x_n}$.
\end{theorem}

\begin{proof}
We first record the elementary identity that relates hitting to new entries.  For every integer $t\ge0$,
\[
        \{T_{A_n}>t\}
        =
        \{X_n(0)\notin A_n,\ E_1=\cdots=E_t=0\}.
\]
Thus, under stationarity, the only difference between ``no hit'' and ``no entry'' is the event $X_n(0)\in A_n$, whose probability is $\pi_n(A_n)=o(1)$.

We begin with the stationary initial law.  Fix $s>0$ and set
\[
        m_n:=\floor{\frac{s}{b_n\mu_n}}.
\]
Use $m_n$ active blocks of length $b_n$, separated by gaps of length $h_n$:
\[
        I_\ell:=\{(\ell-1)(b_n+h_n)+1,\ldots,(\ell-1)(b_n+h_n)+b_n\},
        \qquad 1\le \ell\le m_n.
\]
Let $B_\ell$ be the event that at least one new entry occurs in the active block $I_\ell$, and put
\[
        \beta_n:=\Pbb_{\pi_n}(B_1).
\]
By stationarity and assumption \textnormal{(M1)},
\[
        \beta_n=b_n\mu_n(1+o(1)).
\]

We use the following elementary consequence of the definition of $\alpha_n$.  If $\eta$ is any probability law on the state space, then
\[
        \|\eta P_n^{h_n}-\pi_n\|_{\TV}
        \le \sup_y\|P_n^{h_n}(y,\cdot)-\pi_n\|_{\TV}
        =\alpha_n(h_n).
\]
Thus, after any conditioning on the past up to the end of an active block, the law at the beginning of the next active block, after the separating gap of length $h_n$, is within $\alpha_n(h_n)$ of stationarity.  By Lemma~\ref{lem:path-tv-comparison}, the conditional probability of a new entry in the next active block differs from $\beta_n$ by at most $\alpha_n(h_n)$.

Let
\[
        p_{\ell,n}:=
        \Pbb_{\pi_n}\bigl(B_\ell\mid B_1^c\cap\cdots\cap B_{\ell-1}^c\bigr),
        \qquad 1\le \ell\le m_n,
\]
with the convention that $p_{1,n}=\beta_n$.  The preceding paragraph gives, uniformly in $\ell$,
\[
        p_{\ell,n}=\beta_n+\OO(\alpha_n(h_n)).
\]
Hence
\[
        \sum_{\ell=1}^{m_n}p_{\ell,n}
        =m_n\beta_n+\OO(m_n\alpha_n(h_n))=s+o(1),
        \qquad
        \sum_{\ell=1}^{m_n}p_{\ell,n}^2=o(1),
\]
because $m_n\alpha_n(h_n)=o(1)$ and $m_n\beta_n^2=\OO(b_n\mu_n)\to0$.  Therefore
\[
\begin{aligned}
        \Pbb_{\pi_n}(B_1^c\cap\cdots\cap B_{m_n}^c)
        &=\prod_{\ell=1}^{m_n}(1-p_{\ell,n}) \\
        &=\exp\left\{-\sum_{\ell=1}^{m_n}p_{\ell,n}
            +\OO\left(\sum_{\ell=1}^{m_n}p_{\ell,n}^2\right)\right\} \\
        &\longrightarrow e^{-s}.
\end{aligned}
\]

Entries occurring in the separating gaps are negligible.  The total number of gap times is at most $m_nh_n$, so stationarity gives expected gap entries at most
\[
        m_nh_n\mu_n
        \le \frac{s}{b_n\mu_n}h_n\mu_n+o(1)
        =s\frac{h_n}{b_n}+o(1)\to0.
\]
Thus, by Markov's inequality, the probability of any gap entry tends to zero.

The total length of the active-block-plus-gap construction is
\[
        m_n(b_n+h_n)=\frac{s}{\mu_n}(1+o(1)).
\]
Because $\pi_n(A_n)\to0$, the possibility of starting inside $A_n$ is negligible.  To pass from the block endpoint to the exact time $\floor{s/\mu_n}$, fix $\varepsilon\in(0,s)$.  Applying the preceding block argument with $s-\varepsilon$ and $s+\varepsilon$ gives block endpoints at
\[
        \frac{s-\varepsilon}{\mu_n}(1+o(1))
        \quad\text{and}\quad
        \frac{s+\varepsilon}{\mu_n}(1+o(1)).
\]
For all sufficiently large $n$, these endpoints bracket $\floor{s/\mu_n}$.  Monotonicity of $t\mapsto \Pbb_{\pi_n}(T_{A_n}>t)$ therefore gives
\[
        e^{-(s+\varepsilon)}+o(1)
        \le
        \Pbb_{\pi_n}\left(T_{A_n}>\floor{s/\mu_n}\right)
        \le
        e^{-(s-\varepsilon)}+o(1).
\]
Letting $\varepsilon\downarrow0$ yields
\[
        \Pbb_{\pi_n}\left(T_{A_n}>\floor{s/\mu_n}\right)\to e^{-s}.
\]

We now pass from stationarity to the deterministic initial state $x_n$.  Assumption (M2) gives a burn-in time $r_n$ for which the chain has not hit $A_n$ with probability $1-o(1)$, has mixed to $\pi_n$ in total variation, and satisfies $r_n\mu_n=o(1)$.  Let
\[
        H_n:=\{T_{A_n}>r_n\}.
\]
Let $\lambda_n$ be the law of $X_n(r_n)$ under $\Pbb_{x_n}$, and let $\lambda_n^H$ be the law of $X_n(r_n)$ conditioned on $H_n$.  Since $\Pbb_{x_n}(H_n^c)=o(1)$,
\[
        \|\lambda_n^H-\lambda_n\|_{\TV}
        \le \frac{\Pbb_{x_n}(H_n^c)}{\Pbb_{x_n}(H_n)}=o(1).
\]
Together with $\|\lambda_n-\pi_n\|_{\TV}=o(1)$, this shows that, on the no-hit event, the post-burn-in law is still $o(1)$-close to stationarity.  By Lemma~\ref{lem:path-tv-comparison}, starting the future path from this conditional law changes any post-burn-in path event by only $o(1)$ relative to a stationary start.  The time $r_n$ itself is negligible on the $\mu_n^{-1}$ scale.  Indeed, if
\[
        u_n:=\floor{s/\mu_n}-r_n,
\]
then $u_n\ge0$ for all sufficiently large $n$ and $\mu_nu_n\to s$.  The same bracketing argument used above gives the stationary survival asymptotic along any integer sequence with scaled length tending to $s$.  Applying that asymptotic to this post-burn-in horizon gives, for every fixed $s>0$,
\[
        \Pbb_{x_n}\left(T_{A_n}>\floor{s/\mu_n}\right)\to e^{-s}.
\]
This is equivalent to $\mu_nT_{A_n}\Rightarrow\Exp(1)$.
\end{proof}

\begin{remark}[A standard one-block verification]
In applications below, assumption \textnormal{(M1)} is usually verified by proving a first-moment identity and a negligible second factorial moment.  If
\[
        \Ebb_{\pi_n}N_n^{(b)}=b_n\mu_n
        \quad\text{and}\quad
        \Ebb_{\pi_n}\bigl[(N_n^{(b)})_2\bigr]=o(b_n\mu_n),
\]
where \((N)_2=N(N-1)\), then
\[
        0\le \Ebb_{\pi_n} N_n^{(b)}-\Pbb_{\pi_n}(N_n^{(b)}\ge1)
        \le \Ebb_{\pi_n}\bigl[(N_n^{(b)})_2\bigr].
\]
Hence \(\Pbb_{\pi_n}(N_n^{(b)}\ge1)=b_n\mu_n(1+o(1))\).  In this form the second factorial moment is precisely the one-block clump-control estimate: after one entry into the target set, additional entries in the same block must be negligible relative to the first entry rate.
\end{remark}

\begin{remark}[Interface with monotone terminal clouds]
In monotone coupon collectors, the central object is a terminal cloud $W_n(t)$, and completion is the event $W_n(t)=0$.  In the present theorem, the target set $A_n$ is not absorbing and visits occur in clumps.  The relevant intensity is therefore the stationary \emph{entry} rate $\mu_n$, rather than the stationary mass $\pi_n(A_n)$ alone.  When visits to $A_n$ last for an average clump length $\ell_n$, one typically has $\mu_n\asymp \pi_n(A_n)/\ell_n$.
\end{remark}

\section{The reset-button collector: exact regeneration}

The reset-button collector is not another stationary-entry application.  A reset
returns the process exactly to the empty state, so the path decomposes into
independent excursions.  Completion occurs when one excursion reaches the
all-present state before the next reset.  Thus the controlling quantity is a
successful-excursion probability, not a stationary-entry flux into completion.

\subsection{The reset-button model}

Fix $n\ge1$.  There are standard coupons $1,\ldots,n$ and a reset coupon $\otimes$.  At each discrete time one coupon is sampled.  The reset coupon appears with probability
\[
        \rho_n\in(0,1),
\]
and standard coupon $i$ appears with probability $p_{i,n}>0$, where
\[
        \sum_{i=1}^n p_{i,n}=q_n:=1-\rho_n.
\]
When $\otimes$ appears, the current collection is emptied, and the collector then continues from the empty collection.  Let $T_n$ be the first time at which all $n$ standard coupons are present.

It is useful to separate calendar time from standard-coupon time.  Conditional on a draw being standard, coupon $i$ has probability
\[
        \pi_{i,n}:=\frac{p_{i,n}}{q_n},
        \qquad 1\le i\le n.
\]
Let $C_n$ denote the ordinary coupon collector completion time with coupon probabilities $\pi_{1,n},\ldots,\pi_{n,n}$, measured in standard-coupon draws.  Its probability generating function is
\[
        \phi_n(z):=\Ebb z^{C_n}.
\]

\begin{remark}[Connection with Jockovi\'c--Todi\'c]
Jockovi\'c and Todi\'c analyze this model under the notation CCPRB\@.  Their Theorem~1 gives a finite distribution formula by decomposing a sample path into standard-coupon blocks separated by reset symbols.  Their Section~3.2 treats the continuing-after-reset case by a Markov chain on the number of distinct standard coupons currently held, and their Theorem~3 gives a beta-function formula for the expected waiting time when the standard coupons are equally likely \cite{JockovicTodic2024}.  They also discuss expectation asymptotics in several reset regimes.  The purpose of the present section is different: we make the exact regeneration explicit.  Each reset-free excursion is an ordinary coupon-collector attempt, and the number of failed excursions before completion is geometric with success probability
\[
        s_n=\Ebb q_n^{C_n},
\]
where $C_n$ is the ordinary coupon-collector completion time.  This representation recovers the beta-function expectation as a corollary and gives the full probability generating function, exponential rare-success limits, fixed-moment convergence, and positive exponential-moment convergence.
\end{remark}

\subsection{Exact regeneration}

\begin{lemma}[One-excursion structure]\label{lem:reset-one-excursion}
Let $K_n$ be the number of standard-coupon draws before the next reset.  Then
\[
        \Pbb(K_n=k)=\rho_n q_n^k,
        \qquad k=0,1,2,\ldots.
\]
A reset-free excursion succeeds precisely when
\[
        C_n\le K_n.
\]
Consequently the one-excursion success probability is
\[
        s_n:=\Pbb(C_n\le K_n)=\Ebb q_n^{C_n}=\phi_n(q_n).
\]
\end{lemma}

\begin{proof}
The event $\{K_n=k\}$ means that the first $k$ draws are standard and the next draw is a reset.  Hence its probability is $q_n^k\rho_n$.

Given that the first $k$ draws are standard, their conditional types are independent with probabilities $\pi_{1,n},\ldots,\pi_{n,n}$.  The reset indicators are independent of these conditional standard-coupon labels.  Thus the number of standard draws needed for completion during the excursion may be taken as an independent copy of $C_n$.  The excursion succeeds exactly when completion occurs before the reset, namely when $C_n\le K_n$.  Therefore
\[
        \Pbb(C_n\le K_n\mid C_n)=\Pbb(K_n\ge C_n\mid C_n)=q_n^{C_n}.
\]
Taking expectations gives $s_n=\Ebb q_n^{C_n}$.
\end{proof}

\begin{proposition}[Exact regenerative representation]\label{prop:reset-regen-representation}
Let $(C_{n,r},K_{n,r})_{r\ge1}$ be independent copies of $(C_n,K_n)$, with $C_{n,r}$ independent of $K_{n,r}$ for each $r$.  Let
\[
        G_n:=\inf\{r\ge1:C_{n,r}\le K_{n,r}\}.
\]
Then $G_n$ is geometric on $\{1,2,\ldots\}$ with success probability $s_n$, and
\[
        T_n\stackrel{d}=\sum_{r=1}^{G_n-1}(K_{n,r}+1)+C_{n,G_n},
\]
where the sum is empty if $G_n=1$.
\end{proposition}

\begin{proof}
Each reset returns the process to the empty collection, so the post-reset future is independent of the past and has the same law as the original process.  Therefore successive reset-free excursions are independent.  By Lemma~\ref{lem:reset-one-excursion}, each excursion succeeds with probability $s_n$.  Hence the index of the first successful excursion is geometric with parameter $s_n$.

The stopped iid construction above automatically restricts the excursions with $r<G_n$ to be failures and the excursion with $r=G_n$ to be the first success.  If an excursion fails, it consists of $K_{n,r}$ standard draws followed by one reset draw, and therefore has calendar length $K_{n,r}+1$.  In the first successful excursion, the process stops at the ordinary completion time $C_{n,G_n}$ before the next reset is observed.  Summing the failed-excursion lengths and the final successful length gives the displayed representation.
\end{proof}

\subsection{Exact probability generating function and expectation}

\begin{theorem}[Exact probability generating function]\label{thm:reset-pgf}
For $|z|\le1$,
\[
        \Ebb z^{T_n}
        =\frac{(1-q_nz)\phi_n(q_nz)}{1-z+\rho_nz\phi_n(q_nz)}.
\]
The same identity also holds for any real $z\ge0$ such that $q_nz<1$ and
\[
        1-z+\rho_nz\phi_n(q_nz)>0.
\]
In that case the displayed value is finite.
\end{theorem}

\begin{proof}
Let $F_n(z):=\Ebb z^{T_n}$.  In the first excursion, success contributes
\[
        \Ebb\left[z^{C_n}\1_{\{C_n\le K_n\}}\right]
        =\Ebb z^{C_n}q_n^{C_n}
        =\phi_n(q_nz).
\]
Failure after exactly $k$ standard draws contributes probability $\rho_nq_n^k\Pbb(C_n>k)$, calendar factor $z^{k+1}$, and then an independent fresh copy of the full process.  Therefore
\[
        F_n(z)
        =\phi_n(q_nz)
         +\rho_n z F_n(z)\sum_{k\ge0}(q_nz)^k\Pbb(C_n>k).
\]
For an integer-valued $C_n\ge1$,
\[
        \sum_{k\ge0}x^k\Pbb(C_n>k)
        =\Ebb\sum_{k=0}^{C_n-1}x^k
        =\frac{1-\phi_n(x)}{1-x}.
\]
Substituting $x=q_nz$ gives
\[
        F_n(z)=\phi_n(q_nz)
        +\rho_nzF_n(z)\frac{1-\phi_n(q_nz)}{1-q_nz}.
\]
For $|z|\le1$ this identity is valid by absolute convergence.  If $z\ge0$, $q_nz<1$, and
\[
        \rho_nz\frac{1-\phi_n(q_nz)}{1-q_nz}<1,
\]
then the same first-excursion expansion is a convergent geometric series.  This inequality is equivalent to
\[
        1-q_nz-\rho_nz+\rho_nz\phi_n(q_nz)>0.
\]
Solving for $F_n(z)$ yields
\[
        F_n(z)
        =\frac{(1-q_nz)\phi_n(q_nz)}
        {1-q_nz-\rho_nz+\rho_nz\phi_n(q_nz)}.
\]
Since $q_n+\rho_n=1$, the denominator is $1-z+\rho_nz\phi_n(q_nz)$.
\end{proof}

\begin{corollary}[Exact expectation]\label{cor:reset-exact-mean}
Let
\[
        s_n=\phi_n(q_n)=\Ebb q_n^{C_n}.
\]
Then
\[
        \Ebb T_n=\frac{1-s_n}{\rho_ns_n}.
\]
\end{corollary}

\begin{proof}
Differentiate the probability generating function in Theorem~\ref{thm:reset-pgf} at $z=1$.  Write $s_n=\phi_n(q_n)$.  The numerator and denominator are
\[
        N_n(z)=(1-q_nz)\phi_n(q_nz),
        \qquad
        D_n(z)=1-z+\rho_nz\phi_n(q_nz).
\]
At $z=1$,
\[
        N_n(1)=D_n(1)=\rho_ns_n.
\]
Also
\[
        N_n'(1)=-q_ns_n+\rho_nq_n\phi_n'(q_n),
\]
and
\[
        D_n'(1)=-1+\rho_ns_n+\rho_nq_n\phi_n'(q_n).
\]
Thus
\[
        \Ebb T_n=F_n'(1)
        =\frac{N_n'(1)-D_n'(1)}{\rho_ns_n}
        =\frac{1-s_n}{\rho_ns_n}.
\]
\end{proof}

\begin{remark}[A direct cycle proof]
The same formula also follows from first-step regeneration.  The probability that one excursion succeeds is $s_n$.  The exact PGF proof is more useful because it also gives limiting distributions and higher moments.
\end{remark}

\subsection{Equal standard probabilities and the beta-function formula}

Assume now that all standard coupons have the same probability,
\[
        p_{i,n}=\frac{q_n}{n},
        \qquad 1\le i\le n.
\]
Then the conditional standard probabilities are uniform.

\begin{lemma}[Laplace transform of the ordinary uniform collector]\label{lem:reset-uniform-laplace}
Let $C_n$ be the ordinary uniform coupon collector time measured in standard-coupon draws.  Then, for $0\le x<1$,
\[
        \phi_n(x)=\prod_{j=1}^n\frac{(j/n)x}{1-(1-j/n)x}.
\]
Consequently, with $q_n=1-\rho_n$,
\[
        s_n=\phi_n(q_n)
        =\prod_{j=1}^n\frac{j}{j+n\rho_n/q_n}
        =\frac{\Gamma(n+1)\Gamma(1+n\rho_n/q_n)}
        {\Gamma(n+1+n\rho_n/q_n)}.
\]
\end{lemma}

\begin{proof}
In the ordinary uniform coupon collector, after $j-1$ distinct coupons have been collected, the probability that the next standard draw gives a new coupon is $(n-j+1)/n$.  Hence
\[
        C_n\stackrel d=\sum_{k=1}^n Y_{n,k},
\]
where the $Y_{n,k}$ are independent geometric random variables on $\{1,2,\ldots\}$ with success probabilities $k/n$, for $k=1,\ldots,n$.

For such a geometric variable,
\[
        \Ebb x^{Y_{n,k}}
        =\frac{(k/n)x}{1-(1-k/n)x}.
\]
Multiplying over $k=1,\ldots,n$ gives the product formula.  At $x=q_n$,
\[
        \frac{(j/n)q_n}{1-(1-j/n)q_n}
        =\frac{jq_n}{n\rho_n+jq_n}
        =\frac{j}{j+n\rho_n/q_n}.
\]
The gamma-function identity follows by writing the product as
\[
        \prod_{j=1}^n\frac{j}{j+a_n}
        =\frac{\Gamma(n+1)\Gamma(1+a_n)}{\Gamma(n+1+a_n)},
        \qquad a_n:=\frac{n\rho_n}{q_n}.
\]
\end{proof}

\begin{corollary}[Recovery of the Jockovi\'c--Todi\'c expectation]\label{cor:reset-JT-mean}
In the equal-probability reset-button collector,
\[
        \Ebb T_n
        =\frac1{\rho_n}\left(
        \frac{1}{n\rho_n B\!\left(n,\frac{n\rho_n}{1-\rho_n}\right)}-1
        \right).
\]
\end{corollary}

\begin{proof}
Put
\[
        a_n:=\frac{n\rho_n}{q_n}.
\]
By Lemma~\ref{lem:reset-uniform-laplace},
\[
        s_n=\frac{\Gamma(n+1)\Gamma(1+a_n)}{\Gamma(n+1+a_n)}.
\]
Since $\Gamma(n+1)=n\Gamma(n)$ and $\Gamma(1+a_n)=a_n\Gamma(a_n)$,
\[
        s_n
        =na_n\frac{\Gamma(n)\Gamma(a_n)}{\Gamma(n+a_n+1)}.
\]
But $n+a_n=n/q_n$, so
\[
        \Gamma(n+a_n+1)=(n+a_n)\Gamma(n+a_n)=\frac{n}{q_n}\Gamma(n+a_n).
\]
Therefore
\[
        s_n
        =a_nq_n\frac{\Gamma(n)\Gamma(a_n)}{\Gamma(n+a_n)}
        =n\rho_n B(n,a_n).
\]
Substituting this into Corollary~\ref{cor:reset-exact-mean} gives
\[
        \Ebb T_n
        =\frac{1-n\rho_nB(n,a_n)}{\rho_n n\rho_nB(n,a_n)}
        =\frac1{\rho_n}\left(\frac{1}{n\rho_nB(n,a_n)}-1\right),
\]
which is the displayed formula.
\end{proof}

\begin{remark}[What is new relative to Jockovi\'c--Todi\'c]
Corollary~\ref{cor:reset-JT-mean} recovers the beta-function expectation formula of Jockovi\'c and Todi\'c.  The new contribution is not the expectation itself, nor the reset-button model or its basic finite block distribution.  The contribution is the structural compression of those blocks into exact regeneration, the resulting one-line PGF, and the distributional consequences below: exponential rare-success limits, fixed-moment convergence, and positive exponential-moment convergence.
\end{remark}

\subsection{Regenerative rare-success limit}

The next theorem is the regenerative analogue of the stationary-entry hitting theorem proved in Section~\ref{sec:stationary-entry-module}.  It is strictly easier because regeneration at the empty state is exact.  The theorem is stated in a form that separates ordinary moment convergence from the stronger positive-exponential-moment conclusion.

\begin{theorem}[Exact-regeneration rare-success theorem]\label{thm:reset-rare-success}
For each $n$, let $0<\rho_n<1$, $q_n=1-\rho_n$, and let $C_n$ be any positive integer-valued ordinary coupon collector time with probability generating function $\phi_n$.  Define
\[
        s_n:=\phi_n(q_n)=\Ebb q_n^{C_n}.
\]
Assume
\begin{equation}\label{eq:reset-rare-success-assumptions}
        s_n\to0,
        \qquad
        \rho_n q_n\phi_n'(q_n)\to0.
\end{equation}
Then
\[
        \rho_ns_nT_n\Rightarrow \Exp(1),
\]
and, for every fixed integer $r\ge1$,
\[
        \Ebb(\rho_ns_nT_n)^r\longrightarrow r!.
\]
\end{theorem}

\begin{proof}
Let
\[
        \mathcal S_n:=\{C_n\le K_n\}
\]
be the event that a single reset-free excursion succeeds.  Thus $\Pbb(\mathcal S_n)=s_n$.

Use the regenerative representation in Proposition~\ref{prop:reset-regen-representation}, but write it in conditional excursion form.  Let $M_n$ be geometric on $\{0,1,2,\ldots\}$ with
\[
        \Pbb(M_n=m)=s_n(1-s_n)^m,
\]
let $L_{n,1}^{\mathrm{fail}},L_{n,2}^{\mathrm{fail}},\ldots$ be independent copies of
\[
        K_n+1\,\big|\,\mathcal S_n^c,
\]
and let $L_n^{\mathrm{suc}}$ have the law of
\[
        C_n\,\big|\,\mathcal S_n.
\]
Then $M_n$, the failed-excursion lengths, and the final successful-excursion length may be taken independent, and
\begin{equation}\label{eq:reset-conditional-regenerative-decomposition}
        T_n\stackrel{d}=\sum_{i=1}^{M_n}L_{n,i}^{\mathrm{fail}}+L_n^{\mathrm{suc}}.
\end{equation}

Put
\[
        Y_n:=\rho_n L_n^{\mathrm{fail}}.
\]
We first record the two estimates on $Y_n$ needed below.  Since $L_n^{\mathrm{fail}}=K_n+1$ under the failure event and $\Pbb(\mathcal S_n^c)=1-s_n\to1$, for every fixed integer $r\ge1$,
\[
        \Ebb Y_n^r
        \le \frac{\Ebb\bigl[(\rho_n(K_n+1))^r\bigr]}{1-s_n}=O_r(1),
\]
because $K_n$ is geometric with parameter $\rho_n$.  Moreover,
\[
        \Ebb Y_n
        =\frac{\rho_n\Ebb[(K_n+1)\1_{\mathcal S_n^c}]}{1-s_n}
        =\frac{1-\rho_n\Ebb[(K_n+1)\1_{\mathcal S_n}]}{1-s_n}.
\]
Conditioning on $C_n=c$ gives
\[
        \Ebb[(K_n+1)\1_{\{K_n\ge c\}}]
        =q_n^c\left(c+1+\frac{q_n}{\rho_n}\right).
\]
Therefore
\[
        \rho_n\Ebb[(K_n+1)\1_{\mathcal S_n}]
        =\rho_n\Ebb[(C_n+1)q_n^{C_n}]+q_n\Ebb q_n^{C_n}
        =\rho_nq_n\phi_n'(q_n)+s_n.
\]
By \eqref{eq:reset-rare-success-assumptions},
\[
        \Ebb Y_n\to1.
\]

Now consider the failed-excursion contribution
\[
        U_n:=\rho_ns_n\sum_{i=1}^{M_n}L_{n,i}^{\mathrm{fail}}
        =s_n\sum_{i=1}^{M_n}Y_{n,i}.
\]
For fixed $\theta\ge0$,
\[
        \Ebb e^{-\theta s_nY_n}=1-\theta s_n\Ebb Y_n+o(s_n)
        =1-\theta s_n+o(s_n),
\]
where the $o(s_n)$ term follows from the uniform boundedness of the second moments of $Y_n$.  Since $M_n$ is geometric on $\{0,1,2,\ldots\}$,
\[
\begin{aligned}
        \Ebb e^{-\theta U_n}
        &=\frac{s_n}{1-(1-s_n)\Ebb e^{-\theta s_nY_n}}  \\
        &\longrightarrow \frac1{1+\theta}.
\end{aligned}
\]
Thus $U_n\Rightarrow\Exp(1)$.

The same decomposition gives moment convergence.  Indeed, for fixed $r\ge1$, expand
\[
        \left(s_n\sum_{i=1}^{M_n}Y_{n,i}\right)^r
\]
according to the number of distinct indices appearing in the product.  The term with $r$ distinct indices equals
\[
        s_n^r\Ebb[(M_n)_r](\Ebb Y_n)^r
        =s_n^r\frac{r!(1-s_n)^r}{s_n^r}(\Ebb Y_n)^r
        \longrightarrow r!.
\]
All terms using only $k<r$ distinct indices are bounded by
\[
        C_r s_n^r\Ebb[M_n^k]=O_r(s_n^{r-k})\to0,
\]
because the fixed moments of $Y_n$ are uniformly bounded and, for a geometric variable with parameter $s_n$, $\Ebb M_n^k=O_k(s_n^{-k})$ for each fixed $k$.  Hence
\begin{equation}\label{eq:reset-failed-moment-convergence}
        \Ebb U_n^r\to r!,
        \qquad r\ge1\text{ fixed}.
\end{equation}

It remains to show that the final successful excursion is negligible.  Put
\[
        V_n:=\rho_ns_nL_n^{\mathrm{suc}}.
\]
For the first moment,
\[
        \Ebb V_n
        =\rho_ns_n\Ebb[C_n\mid \mathcal S_n]
        =\rho_n\Ebb[C_nq_n^{C_n}]
        =\rho_nq_n\phi_n'(q_n)\to0.
\]
For $r\ge2$, using $C_n\le K_n$ on the event \(\mathcal S_n\),
\[
        \Ebb V_n^r
        =(\rho_ns_n)^r\frac{\Ebb[C_n^r\1_{\mathcal S_n}]}{s_n}
        \le (\rho_ns_n)^r\frac{\Ebb K_n^r}{s_n}
        \le C_r s_n^{r-1}\to0.
\]
Thus $V_n\to0$ in every fixed $L^r$.  Combining this with \eqref{eq:reset-conditional-regenerative-decomposition} and \eqref{eq:reset-failed-moment-convergence} proves both the distributional convergence and the fixed-moment convergence for $\rho_ns_nT_n=U_n+V_n$.
\end{proof}

\begin{corollary}[Positive exponential moments]\label{cor:reset-positive-exponential-moments}
Assume the hypotheses of Theorem~\ref{thm:reset-rare-success}.  If, for a fixed $a\in(0,1)$,
\begin{equation}\label{eq:reset-right-tilt-condition}
        \frac{\phi_n(q_ne^{a\rho_ns_n})}{s_n}\longrightarrow1,
\end{equation}
then
\[
        \Ebb e^{a\rho_ns_nT_n}\longrightarrow\frac1{1-a}.
\]
If \eqref{eq:reset-right-tilt-condition} holds locally uniformly for $a$ in compact subintervals of $(0,1)$, then the positive exponential moments converge locally uniformly there.
\end{corollary}

\begin{proof}
Let $F_n(z)=\Ebb z^{T_n}$.  By Theorem~\ref{thm:reset-pgf},
\[
        F_n(z)=\frac{(1-q_nz)\phi_n(q_nz)}{1-z+\rho_nz\phi_n(q_nz)}.
\]
Put $z_n=e^{a\rho_ns_n}$.  For all large $n$, $q_nz_n<1$, since
\[
        q_nz_n=(1-\rho_n)(1+a\rho_ns_n+o(\rho_ns_n))<1.
\]
The right-tilt condition gives
\[
        \phi_n(q_nz_n)=s_n(1+o(1)).
\]
Also
\[
        1-z_n=-a\rho_ns_n(1+o(1)),
        \qquad
        1-q_nz_n=\rho_n(1+o(1)),
\]
because $s_n\to0$.  Hence
\[
        1-z_n+\rho_nz_n\phi_n(q_nz_n)
        =(1-a)\rho_ns_n(1+o(1))>0
\]
for all sufficiently large $n$.  The extended form of Theorem~\ref{thm:reset-pgf} is therefore applicable at $z=z_n$, and substitution gives
\[
        \Ebb e^{a\rho_ns_nT_n}=F_n(z_n)\to\frac1{1-a}.
\]
The locally uniform statement is identical when the right-tilt convergence is locally uniform.
\end{proof}

\begin{remark}[Interpretation of the derivative condition]
The quantity
\[
        q_n\phi_n'(q_n)=\Ebb[C_nq_n^{C_n}]
\]
is the unnormalized first moment of the successful excursion length.  Thus
\[
        \rho_nq_n\phi_n'(q_n)\to0
\]
says that the final successful excursion is negligible on the rare scale $(\rho_ns_n)^{-1}$.  This is the natural regeneration analogue of the burn-in and one-block clump-control conditions in the stationary-entry hitting theorem.
\end{remark}

\begin{remark}[Reset asymptotic regimes]
The equal-probability reset regimes and the negligible-reset return to the ordinary Gumbel law are useful consequences of the regenerative representation, but they are not needed for the stationary-entry applications.  They are collected in Appendix~\ref{app:reset-asymptotic-regimes}.
\end{remark}

\subsection{Unequal probabilities}

The exact regenerative formulas are not restricted to the uniform case.  Let the conditional standard probabilities be
\[
        \pi_{i,n}=\frac{p_{i,n}}{q_n},
        \qquad 1\le i\le n.
\]
Let $C_{\pi,n}$ be the ordinary unequal-probability coupon collector completion time in standard-coupon draws, and set
\[
        \phi_{\pi,n}(z):=\Ebb z^{C_{\pi,n}},
        \qquad
        s_{\pi,n}:=\phi_{\pi,n}(q_n).
\]

\begin{corollary}[Unequal-probability reset collector]\label{cor:reset-unequal}
For arbitrary standard coupon probabilities $p_{1,n},\ldots,p_{n,n}$ summing to $q_n=1-\rho_n$, the reset-button completion time has probability generating function
\[
        \Ebb z^{T_n}
        =\frac{(1-q_nz)\phi_{\pi,n}(q_nz)}
        {1-z+\rho_nz\phi_{\pi,n}(q_nz)}.
\]
In particular,
\[
        \Ebb T_n=\frac{1-s_{\pi,n}}{\rho_ns_{\pi,n}}.
\]
If
\[
        s_{\pi,n}\to0,
        \qquad
        \rho_nq_n\phi_{\pi,n}'(q_n)\to0,
\]
then
\[
        \rho_ns_{\pi,n}T_n\Rightarrow\Exp(1).
\]
Moreover, for every fixed integer $r\ge1$,
\[
        \Ebb(\rho_ns_{\pi,n}T_n)^r\to r!.
\]
If, for a fixed $a\in(0,1)$, the corresponding right-tilt condition
\[
        \frac{\phi_{\pi,n}(q_ne^{a\rho_ns_{\pi,n}})}{s_{\pi,n}}\to1
\]
also holds, then $\Ebb e^{a\rho_ns_{\pi,n}T_n}\to(1-a)^{-1}$.
\end{corollary}

\begin{proof}
This is exactly Theorem~\ref{thm:reset-pgf}, Corollary~\ref{cor:reset-exact-mean}, and Theorem~\ref{thm:reset-rare-success}, with $C_n=C_{\pi,n}$.
\end{proof}

\begin{remark}[Relation to the finite block formula]
Jockovi\'c and Todi\'c give a finite distribution formula for the unequal-probability case by summing over all reset-separated block lengths.  Corollary~\ref{cor:reset-unequal} is not a competing derivation of that finite formula; it is a regenerative compression of the same block decomposition into a one-line PGF identity.  Once the ordinary unequal collector Laplace transform is estimated, the same representation also supplies rare-success asymptotic laws.
\end{remark}

\section{The clumsy coupon collector}

\subsection{Model and stationary-entry flux}

The clumsy coupon collector was introduced by Aggarwal and Garoni \cite{AggarwalGaroni}.  Fix $p\in(0,1)$ and put $q=1-p$.  In the one-set clumsy coupon collector, each step selects one coupon type uniformly from $[n]$.  If type $i$ is selected, then its state is refreshed to present with probability $q$ and absent with probability $p$.  Other types are unchanged.

Let
\[
        X(t)=(X_1(t),\ldots,X_n(t))\in\{0,1\}^n,
\]
where $X_i(t)=1$ means type $i$ is currently present.  The completion set is
\[
        A_n:=\{(1,\ldots,1)\}.
\]

\begin{proposition}[Stationary law and entry flux]\label{prop:clumsy-flux}
The clumsy chain has stationary law
\[
        \pi_n=\operatorname{Bernoulli}(q)^{\otimes n}.
\]
Moreover, if
\[
        \mu_n:=\Pbb_{\pi_n}(X(0)\notin A_n,\ X(1)\in A_n),
\]
then
\[
        \mu_n=pq^n.
\]
\end{proposition}

\begin{proof}
The chain is random-scan heat-bath dynamics for the product Bernoulli$(q)$ law: when coordinate $i$ is selected, it is replaced by an independent Bernoulli$(q)$ variable.  Since $p,q\in(0,1)$, every coordinate can be refreshed to either value, so the finite chain is irreducible.  Hence $\pi_n=\operatorname{Bernoulli}(q)^{\otimes n}$ is stationary.

Under stationarity,
\[
        \Pbb_{\pi_n}(X(1)\in A_n)=q^n.
\]
Also, conditional on $X(0)\in A_n$, the chain remains in $A_n$ at the next step exactly when the selected coordinate is refreshed to $1$, an event of probability $q$.  Therefore
\[
\begin{aligned}
        \mu_n
        &=\Pbb_{\pi_n}(X(1)\in A_n)-\Pbb_{\pi_n}(X(0)\in A_n,X(1)\in A_n) \\
        &=q^n-q^{n+1}=pq^n.
\end{aligned}
\]
\end{proof}

\subsection{One-block clump control}

\begin{lemma}[Clumsy one-block clump control]\label{lem:clumsy-clump}
Let $b_n\le n^M$ for a fixed $M<\infty$.  Let
\[
        E_t=\1\{X(t-1)\notin A_n,\ X(t)\in A_n\},
        \qquad
        N_b=\sum_{t=1}^{b_n}E_t.
\]
Then
\[
        \Pbb_{\pi_n}(N_b\ge1)=b_n p q^n(1+o(1)).
\]
\end{lemma}

\begin{proof}
The first moment is exact:
\[
        \Ebb_{\pi_n}N_b=b_npq^n.
\]
We prove that the second factorial moment is negligible relative to this first moment.

Let
\[
        K(t):=\#\{i:X_i(t)=0\}
\]
be the number of absent types.  Then $A_n=\{K=0\}$.  The process $K(t)$ is a birth-death chain with transition probabilities
\[
        k\to k+1 \quad\text{with probability}\quad \frac{n-k}{n}p,
        \qquad
        k\to k-1 \quad\text{with probability}\quad \frac{k}{n}q,
\]
and otherwise it stays fixed.  Write $\Pbb_k$ and $\Ebb_k$ for the law and expectation of this count chain started from $K(0)=k$.

Choose a small constant $\delta>0$ such that
\[
        \frac{q\delta}{p(1-\delta)}\le\frac12,
\]
and set $m_n=\floor{\delta n}$.  Starting from $K=1$, let
\[
        \tau_0=\inf\{t:K(t)=0\},
        \qquad
        \tau_m=\inf\{t:K(t)=m_n\}.
\]
For $1\le \ell\le m_n-1$, the ratio of downward to upward birth-death probabilities is
\[
        r_\ell=\frac{d_\ell}{u_\ell}
        =\frac{(\ell/n)q}{((n-\ell)/n)p}
        =\frac{q\ell}{p(n-\ell)}.
\]
Thus $r_\ell\le1/2$ for $1\le\ell\le m_n-1$, and $r_1=\OO(n^{-1})$.
We recall the elementary derivation of the required hitting probability.  Holding probabilities do not affect hitting probabilities.  If
\[
        h(i)=\Pbb_i(\tau_0<\tau_m),
        \qquad 0\le i\le m,
\]
for a birth-death chain on $\{0,1,\ldots,m\}$ with upward probabilities $u_i$ and downward probabilities $d_i$, then $h(0)=1$, $h(m)=0$, and
\[
        u_i\bigl(h(i+1)-h(i)\bigr)
        =d_i\bigl(h(i)-h(i-1)\bigr),
        \qquad 1\le i\le m-1.
\]
Iterating this first-difference recursion and imposing the two boundary values gives, starting from $1$,
\[
        \Pbb_1(\tau_0<\tau_m)
        =
        \frac{\sum_{j=1}^{m-1}\prod_{\ell=1}^{j} d_\ell/u_\ell}
        {1+\sum_{j=1}^{m-1}\prod_{\ell=1}^{j} d_\ell/u_\ell}.
\]
Applying this with $m=m_n$ and $d_\ell/u_\ell=r_\ell$ gives
\[
\begin{aligned}
        \Pbb_1(\tau_0<\tau_m)
        &\le \sum_{j=1}^{m_n-1}\prod_{\ell=1}^j r_\ell  \\
        &\le r_1\sum_{j=1}^{m_n-1}2^{-(j-1)}
        =\OO(n^{-1}).
\end{aligned}
\]
Let $\theta_n:=\Pbb_1(\tau_0<\tau_m)=\OO(n^{-1})$.  Starting from $0$, ignore the initial holding time at $0$; no new entry is counted while the chain remains at $0$, because $E_t$ records only transitions from outside $A_n$ into $A_n$.  When the chain first leaves $0$, it moves to $1$.  Each excursion from $1$ returns to $0$ before reaching $m_n$ with probability $\theta_n$, and otherwise reaches $m_n$ first.  Hence the number of new returns to $0$ before the first visit to $m_n$ is stochastically dominated by a geometric number with mean
\[
        \frac{\theta_n}{1-\theta_n}=\OO(n^{-1}).
\]

After the chain reaches $m_n$, a return to $0$ in a polynomial block is exponentially unlikely.  By the strong Markov property, restart at the first time when $K$ reaches $m_n$, and choose $m_n$ absent coordinates at that time.  Fix a later time $t$.  If one of these coordinates has not been selected since the restart time, then it is still absent, so $K(t)\ne0$.  Otherwise, conditional on the update history, each chosen coordinate is present at time $t$ only if its last refresh since the restart was to $1$.  These last refresh bits are independent and have probability $q$.  Hence
\[
        \Pbb(K(t)=0\mid\text{update history})\le q^{m_n}.
\]
Therefore
\[
        \sup_{k\ge m_n}\Ebb_k\sum_{t=1}^{b_n}E_t
        \le b_nq^{m_n}=e^{-\Omega(n)}.
\]
Consequently
\[
        \Ebb_0\sum_{t=1}^{b_n}E_t=\OO(n^{-1})+e^{-\Omega(n)}=o(1).
\]

Now use stationarity and the Markov property.  Conditional on $E_s=1$, the chain is in $A_n$ at time $s$.  Therefore
\[
\begin{aligned}
        \Ebb_{\pi_n}[(N_b)_2]
        &=2\sum_{1\le s<t\le b_n}\Pbb_{\pi_n}(E_s=1,E_t=1) \\
        &\le 2b_npq^n\,\Ebb_0\sum_{u=1}^{b_n}E_u \\
        &=o(b_npq^n).
\end{aligned}
\]
Since
\[
        0\le \Ebb_{\pi_n} N_b-\Pbb_{\pi_n}(N_b\ge1)\le \Ebb_{\pi_n}[(N_b)_2],
\]
the claim follows.
\end{proof}

\subsection{Mixing and burn-in}

\begin{lemma}[Clumsy mixing and burn-in]\label{lem:clumsy-mixing}
For the clumsy chain,
\[
        \sup_x\|P^t(x,\cdot)-\pi_n\|_{\TV}\le n(1-1/n)^t.
\]
In particular, with $b_n=n^3$ and $h_n=r_n=\ceil{C n^2}$ for any fixed $C>\log(1/q)$, the mixing and burn-in assumptions in Theorem~\ref{thm:stationary-entry} hold for the empty initial state.
\end{lemma}

\begin{proof}
Couple two copies by using the same selected coordinates and the same refresh bits.  Once a coordinate has been selected at least once, the two copies agree at that coordinate.  Therefore the copies have coalesced once all coordinates have been selected.  The probability that some coordinate has not been selected by time $t$ is at most $n(1-1/n)^t$, proving the mixing bound.

For $h_n=\ceil{C n^2}$ with $C>\log(1/q)$, this bound is $o(b_npq^n)$, because $b_npq^n=n^3pq^n$ and
\[
        n(1-1/n)^{h_n}\le n e^{-Cn}=o(n^3pq^n).
\]
Also $h_n=o(b_n)$ and $b_npq^n\to0$.  Starting from the empty state, condition on the sequence of selected coordinates up to a fixed positive time $t$.  If some coordinate has never been selected, then $X(t)\notin A_n$; otherwise the last refresh bits of all $n$ coordinates must all be present, an event of probability $q^n$.  Thus $\Pbb_\emptyset(X(t)\in A_n)\le q^n$ for every $t\ge1$, and the union bound gives
\[
        \Pbb_\emptyset(T_{A_n}\le r_n)\le r_nq^n=o(1).
\]
Finally $r_npq^n\to0$.
\end{proof}

\subsection{Fixed-\texorpdfstring{$p$}{p} hitting law and moments}

\begin{theorem}[Fixed-$p$ clumsy collector]\label{thm:clumsy-final}
For each fixed $p\in(0,1)$, with $q=1-p$, the completion time $T_n$ of the one-set clumsy coupon collector started from the empty collection satisfies
\[
        p q^n T_n\Rightarrow \Exp(1).
\]
\end{theorem}

\begin{proof}
Apply Theorem~\ref{thm:stationary-entry}.  Proposition~\ref{prop:clumsy-flux} gives the stationary-entry flux $\mu_n=pq^n$.  Lemma~\ref{lem:clumsy-clump} gives the stationary one-block law.  Lemma~\ref{lem:clumsy-mixing} gives fast mixing and negligible burn-in from the empty state.  Hence $pq^nT_n\Rightarrow\Exp(1)$.
\end{proof}

\begin{corollary}[Uniform exponential tails and moments]\label{cor:clumsy-moments}
Let $Y_n:=p q^n T_n$ in the fixed-$p$ clumsy collector started from the empty collection.  There are constants $C,c\in(0,\infty)$, depending only on $p$, such that
\[
        \Pbb(Y_n>x)\le C e^{-cx},
        \qquad x\ge0,
\]
for all sufficiently large $n$.  Consequently, for every fixed $r\ge1$,
\[
        \Ebb Y_n^r\longrightarrow \Ebb Z^r=r!,
        \qquad Z\sim\Exp(1).
\]
In particular,
\[
        \Ebb T_n\sim \frac{1}{p q^n},\qquad
        \Var(T_n)\sim \frac{1}{p^2 q^{2n}}.
\]
\end{corollary}

\begin{proof}
Choose
\[
        b_n=n^3,\qquad h_n=\ceil{C_0n^2},
\]
where $C_0>\log(1/q)$, and put $L_n=b_n+h_n$.  By Lemma~\ref{lem:clumsy-mixing},
\[
        \sup_x\|P^{h_n}(x,\cdot)-\pi_n\|_{\TV}=o(b_n p q^n).
\]
By Lemma~\ref{lem:clumsy-clump}, under stationarity the probability that an active block of length $b_n$ contains at least one new entry into the all-present state is
\[
        b_n p q^n(1+o(1)).
\]
Therefore, uniformly over all starting states $x$,
\[
        \Pbb_x(T_n\le L_n)\ge \frac12 b_n p q^n
\]
for all sufficiently large $n$.  Indeed, hits during the initial gap only help.  After the gap, the path-event total-variation comparison from Section~\ref{sec:stationary-entry-module} shows that the probability of an entry during the following active block differs from its stationary value by at most the total-variation distance after the gap.

By the strong Markov property at times $L_n,2L_n,\ldots$, for every integer $m\ge0$,
\[
        \Pbb_x(T_n>mL_n)\le
        \left(1-\frac12 b_n p q^n\right)^m
        \le \exp\{-m b_n p q^n/2\}.
\]
Since $h_n=o(b_n)$, we have $b_n/L_n\to1$.  Hence there are constants $C,c>0$ such that, for all large $n$ and all $x\ge0$,
\[
        \Pbb(pq^nT_n>x)\le C e^{-cx}.
\]
This gives uniform integrability of $Y_n^r$ for every fixed $r\ge1$.  Combining this with Theorem~\ref{thm:clumsy-final}, which gives $Y_n\Rightarrow Z\sim\Exp(1)$, yields moment convergence.  The formulas for the mean and variance follow from the cases $r=1$ and $r=2$.
\end{proof}

\begin{remark}[Relation to the Aggarwal--Garoni fixed-$p$ picture]
Aggarwal and Garoni conjecture exponential order for the mean and variance of the clumsy coupon collector time, together with a Gumbel limit for the standardized completion time \cite{AggarwalGaroni}.  A contemporaneous independent manuscript of Attrill and Garoni, submitted to arXiv concurrently with this work, also proves the fixed-$p$ exponential limit and fixed-$p$ moment asymptotics for the clumsy model by different methods, and treats additional moving-$p$ regimes \cite{AttrillGaroni2026}.  In the one-set fixed-$p$ model considered here, the stationary-entry mechanism gives the conclusion
\[
        p q^nT_n\Rightarrow \Exp(1),
        \qquad
        \Ebb T_n\sim\frac1{pq^n},
        \qquad
        \Var(T_n)\sim\frac1{p^2q^{2n}}.
\]
Consequently
\[
        \frac{T_n-\Ebb T_n}{\sqrt{\Var(T_n)}}
        \Rightarrow \Exp(1)-1,
\]
not a Gumbel law.  Thus any Gumbel behavior, if present, must belong to a different asymptotic regime or a different model convention, rather than to the fixed-$p$ one-set stationary-entry regime analyzed above.
\end{remark}

\section{The careless coupon collector}

\subsection{Model and post-loss convention}

The careless coupon collector was introduced by Cruciani and Dudeja \cite{CrucianiDudeja}.  Fix $p\in(0,1)$ and set $q=1-p$.  In the one-set careless coupon collector, at each round one uniformly random coupon is drawn and added if missing; after this draw, each currently held coupon is independently lost with probability $p$.  We observe the chain after the loss step unless explicitly stated otherwise.

\begin{remark}[Completion convention and relation to Cruciani--Dudeja]
All careless results below use the post-loss convention of Cruciani and Dudeja \cite{CrucianiDudeja}: after the draw, the independent loss step is performed, and the resulting set is the state of the chain.  Equivalently, if $J_t$ is the coupon type drawn at time $t$, their update may be written as
\[
        S_{t+1}=(S_t\cup\{J_t\})\setminus\{i:L_{i,t}=1\}.
\]
Thus $T_{n,p}$ here is the same post-loss hitting time for the all-present state, not a variant in which completion is checked before losses.

If one instead checked completion immediately after the draw but before losses, then the target-entry flux would change by a factor of order $q^{-n}$.  Indeed, from level $n-1$, pre-loss completion would require only drawing the unique missing coupon, whereas post-loss completion additionally requires all $n$ coupons to survive the loss step, contributing the factor $q^n$.  The same stationary-entry theorem applies to that alternative convention, but the normalization must be replaced by the corresponding pre-loss stationary-entry flux.
\end{remark}

Let $S_t\subseteq[n]$ be the set of coupon types held after the loss step at time $t$, and let
\[
        K_t:=|S_t|.
\]
Completion after the loss step is the target event $K_t=n$.

\begin{proposition}[Count-chain transition]\label{prop:careless-transition}
The process $K_t$ is a Markov chain on $\{0,1,\ldots,n\}$. Its transition rule is
\[
K_{t+1}\mid K_t=k\sim
\begin{cases}
\Bin(k,q),&\text{with probability }k/n,\\[1mm]
\Bin(k+1,q),&\text{with probability }(n-k)/n.
\end{cases}
\]
In particular, the chain is irreducible and aperiodic and therefore has a unique stationary law $\nu_n$.
\end{proposition}

\begin{proof}
If $K_t=k$, the next drawn coupon is already present with probability $k/n$, in which case the pre-loss count remains $k$; after losses the next count is $\Bin(k,q)$.  With probability $(n-k)/n$, the drawn coupon is missing, so the pre-loss count is $k+1$; after losses the next count is $\Bin(k+1,q)$.  This proves the transition rule.

The chain can move from any state to $0$ by losing all held coupons, an event of positive probability.  From $0$ it can reach any $j$ by $j$ successive new draws with no losses.  Since $P(0,0)=p>0$, it is aperiodic.
\end{proof}

\begin{remark}[Set-valued irreducibility]
The set-valued careless chain is also irreducible and aperiodic.  From any set, all currently held coupons may be lost in one step with positive probability.  From the empty set, any prescribed set can be built by drawing its elements in succession and losing no held coupon during those steps.  Finally, $P(\emptyset,\emptyset)>0$.  By symmetry, the image of the set-valued stationary law under $S\mapsto |S|$ is the count-chain stationary law $\nu_n$.
\end{remark}

\begin{proposition}[Stationary-entry flux into completion]\label{prop:careless-flux}
Let
\[
        \mu_n:=\Pbb_{\nu_n}(K_0<n,K_1=n)
\]
be the stationary-entry flux into the post-loss completion state.  Then
\[
        \mu_n=\nu_n(n)(1-q^n).
\]
Equivalently,
\[
        \mu_n=\nu_n(n-1)\frac{q^n}{n}.
\]
\end{proposition}

\begin{proof}
Only state $n-1$ can enter state $n$ from outside in one step.  From $n-1$, one must draw the unique missing coupon and lose none of the $n$ coupons; this has probability $q^n/n$.  Thus
\[
        \mu_n=\nu_n(n-1)\frac{q^n}{n}.
\]
Stationarity at state $n$ gives
\[
        \nu_n(n)=\nu_n(n)P(n,n)+\nu_n(n-1)P(n-1,n).
\]
Here $P(n,n)=q^n$ and $P(n-1,n)=q^n/n$.  Hence
\[
        \nu_n(n)=\nu_n(n)q^n+\nu_n(n-1)\frac{q^n}{n},
\]
which proves the two displayed identities.
\end{proof}

\subsection{The stationary high-tail scale and constant}

Define, for $0\le k\le n$,
\[
        w_{n,k}:=\frac{(n)_k}{n^k}q^{k(k+1)/2},
        \qquad
        (n)_k:=n(n-1)\cdots(n-k+1),
\]
with $w_{n,0}=1$.  In particular,
\[
        R_n:=w_{n,n}
        =\frac{n!}{n^n}q^{n(n+1)/2}.
\]
We call $w_{n,k}$ the $k$-step lucky-climb weight: it is the probability of the strictly increasing path $0\to1\to\cdots\to k$, where each step draws a missing coupon and all coupons present before the post-draw loss step survive.  The first high-tail estimate gives the matching scale uniformly throughout the upper tail.  The second theorem identifies the exact constant at the top level.

The proof has three steps.  First, stationary cut-flux gives a uniform comparison $\nu_n(k)\asymp w_{n,k}$ throughout the upper tail.  Second, at fixed levels the finite count chain converges to the infinite immigration-thinning chain $Y_{t+1}\sim\Bin(Y_t+1,q)$, whose stationary law has tail constant $(q;q)_\infty^{-1}$.  Third, normalized cut-flux propagation shows that the ratio $\nu_n(k)/w_{n,k}$ changes by only $1+o(1)$ from a slowly growing level $m_n$ all the way to $n$.

\begin{theorem}[Uniform careless stationary high-tail scale]\label{thm:careless-high-tail}
For each fixed $p\in(0,1)$ there exist constants $0<c_p<C_p<\infty$ and an integer $M_p<\infty$ such that, for all sufficiently large $n$ and all $M_p\le k\le n$,
\[
        c_p w_{n,k}\le \nu_n(k)\le C_p w_{n,k}.
\]
In particular,
\[
        \nu_n(n)=\Theta_p\!\left(\frac{n!}{n^n}q^{n(n+1)/2}\right).
\]
\end{theorem}

\begin{proof}
For $0\le k<n$, the only way to move from $k$ to $k+1$ is to draw a missing coupon and then lose none of the $k+1$ coupons present before losses.  Hence
\[
        u_k:=P(k,k+1)=\frac{n-k}{n}q^{k+1},
\]
and therefore $w_{n,k+1}=w_{n,k}u_k$.

We first record a uniform positive lower bound near zero.  Conditional on $K_t=k$,
\[
        \Ebb[K_{t+1}\mid K_t=k]
        =q\left(k+\frac{n-k}{n}\right)
        =q\left(1+k\left(1-\frac1n\right)\right).
\]
If $m_n=\Ebb_{\nu_n}K$, stationarity gives
\[
        m_n=q\left(1+\left(1-\frac1n\right)m_n\right),
        \qquad
        m_n=\frac{q}{p+q/n}\le\frac{q}{p}.
\]
Choose a fixed integer $M_0=M_0(p)$ with $M_0>4q/p$.  Then
\[
        \nu_n(\{0,1,\ldots,M_0\})\ge\frac34.
\]
From any $j\le M_0$, after the new draw there are at most $M_0+1$ held coupons, and all may be lost with probability at least $p^{M_0+1}$.  Hence stationarity implies
\[
        \nu_n(0)\ge \sum_{j=0}^{M_0}\nu_n(j)P(j,0)
        \ge \frac34p^{M_0+1}=:c_0(p)>0.
\]

Starting from $0$, the strictly increasing path $0\to1\to\cdots\to k$ in exactly $k$ steps has probability
\[
        \prod_{r=0}^{k-1}u_r=w_{n,k}.
\]
Therefore
\[
        \nu_n(k)=\Pbb_{\nu_n}(K_k=k)
        \ge \nu_n(0)w_{n,k}
        \ge c_0(p)w_{n,k}
\]
for every $0\le k\le n$.

For the upper bound, set $A_k=\{0,1,\ldots,k\}$.  Since upward jumps have size at most one, the stationary flux from $A_k$ to $A_k^c$ is exactly $\nu_n(k)u_k$.  By stationarity, this equals the reverse flux:
\[
        \nu_n(k)u_k=\sum_{j=k+1}^n\nu_n(j)P(j,A_k).
\]
Keeping only the $j=k+1$ term gives
\[
        \nu_n(k+1)P(k+1,A_k)\le \nu_n(k)u_k.
\]
Write $\ell=k+1$.  If the drawn coupon is already present, the pre-loss count is $\ell$, and the event $K_{t+1}\ge\ell$ requires all $\ell$ coupons to survive.  If the drawn coupon is missing, the pre-loss count is $\ell+1$, and $K_{t+1}\ge\ell$ requires zero or one loss among these $\ell+1$ coupons.  Therefore
\[
\begin{aligned}
        P(\ell,\{\ell,\ell+1,\ldots,n\})
        &\le \frac{\ell}{n}q^\ell
        +\frac{n-\ell}{n}\left(q^{\ell+1}+(\ell+1)pq^\ell\right)  \\
        &\le C_1(p)\ell q^\ell.
\end{aligned}
\]
Hence
\[
        P(k+1,A_k)\ge 1-C_1(p)(k+1)q^{k+1}.
\]
Choose $M_p\ge M_0$ large enough that $C_1(p)(k+1)q^{k+1}\le1/2$ for all $k\ge M_p$.  Then, for $k\ge M_p$,
\[
        \nu_n(k+1)
        \le \nu_n(k)\frac{u_k}{1-C_1(p)(k+1)q^{k+1}}.
\]
Iterating from $M_p$ to $k-1$ gives
\[
        \nu_n(k)
        \le \nu_n(M_p)\prod_{r=M_p}^{k-1}u_r
        \prod_{r=M_p}^{k-1}\frac{1}{1-C_1(p)(r+1)q^{r+1}}.
\]
The infinite product is finite because $\sum_r rq^r<\infty$.  Also $\nu_n(M_p)\le1$, and $\prod_{r=0}^{M_p-1}u_r$ is bounded below by a positive constant depending only on $p$ for all large $n$.  Therefore $\nu_n(k)\le C_2(p)w_{n,k}$ uniformly for $M_p\le k\le n$.
\end{proof}

Write
\[
        (q;q)_m:=\prod_{r=1}^m(1-q^r),
        \qquad
        (q;q)_\infty:=\prod_{r=1}^{\infty}(1-q^r).
\]
The infinite product is positive because $\sum_r q^r<\infty$.

\begin{lemma}[The limiting infinite count chain]\label{lem:infinite-chain}
Let $Y_t$ be the Markov chain on $\Nbb_0$ defined by
\[
        Y_{t+1}\mid Y_t=k\sim \Bin(k+1,q).
\]
It has a unique stationary law $\pi$.  This law is the distribution of
\[
        Y_\infty=\sum_{r=1}^{\infty}B_r,
        \qquad
        B_r\sim\mathrm{Bernoulli}(q^r)
\]
with the $B_r$ independent.  Moreover,
\[
        \pi(k)\sim \frac{q^{k(k+1)/2}}{(q;q)_\infty}.
\]
\end{lemma}

\begin{proof}
The sum $\sum_{r\ge1}B_r$ is finite almost surely because $\sum_rq^r<\infty$.  Its generating function is
\[
        G(z)=\prod_{r=1}^{\infty}(1-q^r+q^rz).
\]
Suppose $Y_t\stackrel d=\sum_{r\ge1}B_r$.  To form $Y_{t+1}$, add one particle and retain each of the $Y_t+1$ particles independently with probability $q$.  Thus, with independent Bernoulli$(q)$ variables $C_0,C_1,\ldots$,
\[
        Y_{t+1}=C_0+\sum_{r\ge1}B_rC_r.
\]
Now $C_0\sim\mathrm{Bernoulli}(q)$, while $B_rC_r\sim\mathrm{Bernoulli}(q^{r+1})$, and these variables are independent.  Hence this law is stationary.

The displayed law is a stationary probability distribution by direct substitution.  Since the chain is irreducible, positive recurrence follows from the countable-state stationary-distribution criterion \cite[Theorem~21.13]{LevinPeres}; uniqueness of the stationary distribution then follows from irreducibility.

By the $q$-binomial theorem \cite[Chapter~2]{AndrewsPartitions},
\[
        \prod_{r=1}^{\infty}(1+wq^r)
        =\sum_{m=0}^{\infty}\frac{q^{m(m+1)/2}}{(q;q)_m}w^m.
\]
Taking $w=z-1$ and extracting coefficients gives
\[
        \pi(k)=\sum_{m=k}^{\infty}(-1)^{m-k}\binom{m}{k}
        \frac{q^{m(m+1)/2}}{(q;q)_m}.
\]
Set $m=k+\ell$.  Then
\[
\frac{\pi(k)}{q^{k(k+1)/2}}
        =\sum_{\ell=0}^{\infty}(-1)^\ell\binom{k+\ell}{\ell}
        \frac{q^{\ell k+\ell(\ell+1)/2}}{(q;q)_{k+\ell}}.
\]
The $\ell=0$ term is $(q;q)_k^{-1}\to(q;q)_\infty^{-1}$.  The absolute value of the remaining terms is at most
\[
        \frac1{(q;q)_\infty}
        \sum_{\ell=1}^{\infty}\binom{k+\ell}{\ell}q^{\ell k}
        =\frac{(1-q^k)^{-k-1}-1}{(q;q)_\infty},
\]
which tends to zero because $(k+1)q^k\to0$.  This proves the high-tail asymptotic.
\end{proof}

\begin{lemma}[Finite-to-infinite stationary convergence]\label{lem:careless-local-convergence}
For every fixed $M<\infty$,
\[
        (\nu_n(0),\ldots,\nu_n(M))\longrightarrow(\pi(0),\ldots,\pi(M)).
\]
Consequently, there exists a deterministic sequence $m_n\to\infty$ with $m_n=o(\sqrt n)$ such that
\[
        \frac{\nu_n(m_n)}{\pi(m_n)}\to1.
\]
\end{lemma}

\begin{proof}
The stationary mean calculation above gives $\sup_n\Ebb_{\nu_n}K\le q/p$, so the laws $\nu_n$ are tight.  For fixed $k,j$,
\[
        P_n(k,j)=\frac{k}{n}\Pbb(\Bin(k,q)=j)
        +\frac{n-k}{n}\Pbb(\Bin(k+1,q)=j)\longrightarrow \Pbb(\Bin(k+1,q)=j).
\]
Let a subsequence of $\nu_n$ converge weakly to a law $\nu$.  For fixed $j$, stationarity gives
\[
        \nu_n(j)=\sum_{k=0}^{n}\nu_n(k)P_n(k,j).
\]
We justify passage to the limit in this equation.  For fixed $j$ and all sufficiently large $k$,
\[
\begin{aligned}
        P_n(k,j)
        &\le \Pbb(\Bin(k,q)=j)+\Pbb(\Bin(k+1,q)=j)  \\
        &\le C_j(k+1)^j p^{k-j},
\end{aligned}
\]
with $C_j<\infty$ independent of $n$ and $k$.  Hence, for every $K>j$,
\[
        \sup_n\sum_{k>K}\nu_n(k)P_n(k,j)
        \le C_j\sup_{k>K}(k+1)^jp^{k-j}\longrightarrow0.
\]
The same bound applies to the limiting kernel
$P(k,j)=\Pbb(\Bin(k+1,q)=j)$, so
\[
        \sum_{k>K}\nu(k)P(k,j)\longrightarrow0.
\]
For fixed $K$, weak convergence and pointwise kernel convergence give
\[
        \sum_{k=0}^{K}\nu_n(k)P_n(k,j)\to
        \sum_{k=0}^{K}\nu(k)P(k,j).
\]
Letting first $n\to\infty$ along the subsequence and then $K\to\infty$ yields
\[
        \nu(j)=\sum_{k=0}^{\infty}\nu(k)P(k,j),
\]
so $\nu$ is stationary for the infinite chain.  By Lemma~\ref{lem:infinite-chain}, $\nu=\pi$.  Therefore every weak subsequential limit is $\pi$, and the full sequence converges locally.

For the slowly growing window, set
\[
        \Delta_n(M):=\max_{0\le j\le M}\left|\frac{\nu_n(j)}{\pi(j)}-1\right|.
\]
For each fixed $M$, local convergence and the positivity of $\pi(j)$ on the finite set $\{0,\ldots,M\}$ imply $\Delta_n(M)\to0$.  Choose integers $N_M$ such that $n\ge N_M$ implies $\Delta_n(M)\le1/M$.  Increasing them if necessary, assume that $N_M$ is nondecreasing.  Define, for all large $n$,
\[
        m_n:=\max\{M\le n^{1/4}: n\ge N_M\}.
\]
Since each $N_M$ is finite, $m_n\to\infty$; by construction $m_n=o(\sqrt n)$ and $\Delta_n(m_n)\le1/m_n\to0$.  Thus $\nu_n(m_n)/\pi(m_n)\to1$.
\end{proof}

\begin{lemma}[Intermediate normalized tail]\label{lem:intermediate-tail}
With the sequence $m_n$ from Lemma~\ref{lem:careless-local-convergence},
\[
        \frac{\nu_n(m_n)}{w_{n,m_n}}\to \frac1{(q;q)_\infty}.
\]
\end{lemma}

\begin{proof}
Since $m_n=o(\sqrt n)$,
\[
        \frac{(n)_{m_n}}{n^{m_n}}
        =\prod_{r=0}^{m_n-1}\left(1-\frac r n\right)\to1.
\]
Lemma~\ref{lem:infinite-chain} gives
\[
        \pi(m_n)\sim \frac{q^{m_n(m_n+1)/2}}{(q;q)_\infty}.
\]
Together with $\nu_n(m_n)/\pi(m_n)\to1$, this proves the claim.
\end{proof}

\begin{lemma}[Normalized cut-flux propagation]\label{lem:normalized-propagation}
Let
\[
        a_{n,k}:=\frac{\nu_n(k)}{w_{n,k}}.
\]
For every integer sequence $m_n$ satisfying $m_n\to\infty$ and $m_n\le n$ eventually,
\[
        \sup_{m_n\le k\le n}\left|\log\frac{a_{n,k}}{a_{n,m_n}}\right|\to0.
\]
Consequently $a_{n,n}=a_{n,m_n}(1+o(1))$.
\end{lemma}

\begin{proof}
By Theorem~\ref{thm:careless-high-tail}, there are constants $0<c<C<\infty$ and $M<\infty$ such that
\[
        c\le a_{n,k}\le C,
        \qquad M\le k\le n.
\]
For $k\ge M$, the stationary cut-flux identity gives
\[
        \nu_n(k)u_k=\nu_n(k+1)d_k+R_k,
\]
where
\[
        d_k:=P(k+1,\{0,1,\ldots,k\}),
        \qquad
        R_k:=\sum_{j=k+2}^n\nu_n(j)P(j,\{0,1,\ldots,k\}).
\]
Since $w_{n,k+1}=w_{n,k}u_k$,
\[
        a_{n,k}=d_ka_{n,k+1}+\frac{R_k}{w_{n,k+1}}.
\]
As above,
\[
        1-d_k\le C_pkq^k.
\]
Also, using $P(j,A_k)\le1$ and $\nu_n(j)\le Cw_{n,j}$,
\[
        R_k\le C\sum_{j=k+2}^n w_{n,j}.
\]
For $j\ge k+2$,
\[
        \frac{w_{n,j}}{w_{n,k+1}}
        =\prod_{r=k+1}^{j-1}u_r
        \le q^{(j-k-1)(k+2)}.
\]
Consequently,
\[
        \frac{R_k}{w_{n,k+1}}
        \le C\sum_{h=1}^{\infty}q^{h(k+2)}
        =C\frac{q^{k+2}}{1-q^{k+2}}
        =\OO_p(q^k).
\]
Since $a_{n,k}\ge c>0$,
\[
        \rho_{n,k}:=\frac{R_k}{a_{n,k}w_{n,k+1}}=\OO_p(q^k).
\]
For all sufficiently large $k$, both $\rho_{n,k}$ and $1-d_k$ are less than $1/2$.  Therefore
\[
        \frac{a_{n,k+1}}{a_{n,k}}
        =\frac{1-\rho_{n,k}}{d_k}=1+O_p(kq^k)
\]
uniformly for $M\le k\le n-1$.  Since $m_n\to\infty$, for all large $n$ we have $m_n\ge M$.  Thus, for every endpoint $k\in[m_n,n]$, summing logarithms gives
\[
        \left|\log\frac{a_{n,k}}{a_{n,m_n}}\right|
        \le C\sum_{r\ge m_n}rq^r=o(1),
\]
uniformly in $k$.
\end{proof}

\begin{theorem}[Sharp careless stationary high-tail constant]\label{thm:careless-sharp-constant}
For fixed $p\in(0,1)$ and $q=1-p$,
\[
        \nu_n(n)\sim \frac1{(q;q)_\infty}\frac{n!}{n^n}q^{n(n+1)/2}.
\]
Consequently,
\[
        \mu_n=\nu_n(n)(1-q^n)
        \sim \frac1{(q;q)_\infty}\frac{n!}{n^n}q^{n(n+1)/2}.
\]
\end{theorem}

\begin{proof}
Let $m_n$ be the sequence from Lemma~\ref{lem:careless-local-convergence}.  Lemma~\ref{lem:intermediate-tail} gives
\[
        a_{n,m_n}=\frac{\nu_n(m_n)}{w_{n,m_n}}\to\frac1{(q;q)_\infty}.
\]
Lemma~\ref{lem:normalized-propagation} gives $a_{n,n}=a_{n,m_n}(1+o(1))$.  Since $w_{n,n}=n!n^{-n}q^{n(n+1)/2}$, the first claim follows.  The flux identity follows from Proposition~\ref{prop:careless-flux} and $1-q^n\to1$.
\end{proof}

\subsection{Clump control and mixing}

\begin{lemma}[Careless one-block clump control]\label{lem:careless-clump}
Let $b_n\le n^M$ for some fixed $M<\infty$.  Let
\[
        E_t:=\1\{K_{t-1}<n,\ K_t=n\},
        \qquad
        N_b:=\sum_{t=1}^{b_n}E_t.
\]
Then
\[
        \Pbb_{\nu_n}(N_b\ge1)=b_n\mu_n(1+o(1)),
\]
where $\mu_n=\Pbb_{\nu_n}(E_1=1)$.
\end{lemma}

\begin{proof}
By stationarity,
\[
        \Ebb_{\nu_n}N_b=b_n\mu_n.
\]
We prove the second factorial moment is negligible.  For any starting state $x\in\{0,\ldots,n\}$ and any $u\ge1$,
\[
        \Pbb_x(K_u=n)\le q^n.
\]
Indeed, condition on $K_{u-1}$.  A one-step transition to $n$ is impossible from states at most $n-2$, has probability $q^n/n$ from $n-1$, and has probability $q^n$ from $n$.

Since $E_u\le \1\{K_u=n\}$,
\[
        \Ebb_n\sum_{u=1}^{b_n}E_u
        \le \sum_{u=1}^{b_n}\Pbb_n(K_u=n)
        \le b_nq^n=o(1),
\]
because $b_n$ is polynomial and $q^n$ is exponentially small.

Now use stationarity and the Markov property.  Conditional on $E_s=1$, the chain is at state $n$ at time $s$.  Therefore
\[
\begin{aligned}
        \Ebb_{\nu_n}[(N_b)_2]
        &=2\sum_{1\le s<t\le b_n}\Pbb_{\nu_n}(E_s=1,E_t=1) \\
        &\le 2\sum_{s=1}^{b_n}\Pbb_{\nu_n}(E_s=1)\Ebb_n\sum_{u=1}^{b_n}E_u \\
        &\le 2b_n\mu_n\, b_nq^n=o(b_n\mu_n).
\end{aligned}
\]
Finally,
\[
        0\le \Ebb_{\nu_n} N_b-\Pbb_{\nu_n}(N_b\ge1)\le \Ebb_{\nu_n}[(N_b)_2],
\]
so
\[
        \Pbb_{\nu_n}(N_b\ge1)=b_n\mu_n(1+o(1)).
\]
\end{proof}

\begin{lemma}[Careless fast mixing and negligible burn-in]\label{lem:careless-mixing}
Let $\widehat\pi_n$ be the stationary law of the set-valued careless chain; its image under $S\mapsto |S|$ is $\nu_n$.  Then
\[
        \sup_{S\subseteq[n]}\|P^t(S,\cdot)-\widehat\pi_n\|_{\TV}\le nq^t.
\]
Moreover, from the empty initial collection, with
\[
        b_n=n^3,
        \qquad
        h_n=r_n=\ceil{Cn^2},
\]
for any fixed $C>1$, all mixing and burn-in assumptions in Theorem~\ref{thm:stationary-entry} hold for the post-loss completion state $A_n=\{S:S=[n]\}$.
\end{lemma}

\begin{proof}
Couple two set-valued chains using the same drawn coupon and the same loss indicators for every coupon type.  If a coupon type differs between the two copies at time $t$, then the discrepancy can persist for one more round only if the type is not drawn and the present copy is not lost.  This probability is $(1-1/n)q\le q$.  Therefore a fixed initial discrepancy survives $t$ rounds with probability at most $q^t$, and the union bound over $n$ types gives
\[
        \sup_{S,S'}\Pbb(S_t\ne S'_t\mid S_0=S,S_0'=S')\le nq^t.
\]
The total-variation bound follows by the coupling inequality.

Let $\mu_n$ be the stationary-entry flux.  By Theorem~\ref{thm:careless-sharp-constant},
\[
        \mu_n\sim (q;q)_\infty^{-1}R_n,
        \qquad
        R_n=\frac{n!}{n^n}q^{n(n+1)/2}.
\]
Thus
\[
        \log(b_n\mu_n)=-\frac{\abs{\log q}}{2}n^2+\OO(n)
\]
for $b_n=n^3$.  On the other hand,
\[
        \log(nq^{h_n})=\log n-C\abs{\log q}\,n^2+\OO(1).
\]
Since $C>1$, we have $nq^{h_n}=o(b_n\mu_n)$.  Also $h_n=o(b_n)$ and $b_n\mu_n\to0$.

For burn-in from the empty collection, condition on the pre-loss set in a fixed positive round.  Full completion after the loss step requires all $n$ coupons to survive that loss step, so
\[
        \Pbb_\emptyset(S_t=[n])\le q^n,
        \qquad t\ge1.
\]
Therefore
\[
        \Pbb_\emptyset(T_{A_n}\le r_n)
        \le r_nq^n=o(1).
\]
The same mixing bound gives $\|P^{r_n}(\emptyset,\cdot)-\widehat\pi_n\|_{\TV}=o(1)$, and $r_n\mu_n=o(1)$ because $\mu_n$ is exponentially small in $n^2$.
\end{proof}

\subsection{Fixed-\texorpdfstring{$p$}{p} hitting law and moments}

\begin{theorem}[Fixed-$p$ careless collector]\label{thm:careless-final}
Fix $p\in(0,1)$ and put $q=1-p$.  Let $T_{n,p}$ be the post-loss completion time of the careless coupon collector started from the empty collection, and let
\[
        \mu_n:=\Pbb_{\nu_n}(K_0<n,K_1=n).
\]
Then
\[
        \mu_nT_{n,p}\Rightarrow\Exp(1).
\]
Moreover,
\[
        \mu_n\sim \frac{1}{(q;q)_\infty}
        \frac{n!}{n^n}q^{n(n+1)/2}.
\]
\end{theorem}

\begin{proof}
The target set is $A_n=\{S:S=[n]\}$, equivalently $K=n$.  Theorem~\ref{thm:careless-sharp-constant} identifies the stationary-entry flux $\mu_n$ to sharp asymptotic order.  Since the entry event $\{K_{t-1}<n,K_t=n\}$ depends only on the count coordinate, Lemma~\ref{lem:careless-clump} gives the corresponding stationary one-block law for the set-valued target.  Lemma~\ref{lem:careless-mixing} gives the required mixing and burn-in hypotheses.  Therefore Theorem~\ref{thm:stationary-entry} applies and yields $\mu_nT_{n,p}\Rightarrow\Exp(1)$.
\end{proof}

\begin{corollary}[Careless uniform exponential tails and moments]\label{cor:careless-moments}
Let $Y_n:=\mu_nT_{n,p}$ in the fixed-$p$ careless collector started from the empty collection, where
\[
        \mu_n=\Pbb_{\nu_n}(K_0<n,K_1=n)
\]
is the stationary-entry flux into the post-loss completion state.  There are constants $C,c\in(0,\infty)$, depending only on $p$, such that
\[
        \Pbb(Y_n>x)\le C e^{-cx},
        \qquad x\ge0,
\]
for all sufficiently large $n$.  Consequently, for every fixed $r\ge1$,
\[
        \Ebb Y_n^r\longrightarrow \Ebb Z^r=r!,
        \qquad Z\sim\Exp(1).
\]
In particular,
\[
        \Ebb T_{n,p}\sim \frac{1}{\mu_n},\qquad
        \Var(T_{n,p})\sim \frac{1}{\mu_n^2}.
\]
Together with Theorem~\ref{thm:careless-sharp-constant}, this gives the explicit constant
\[
        \mu_n\sim \frac{1}{(q;q)_\infty}
        \frac{n!}{n^n}q^{n(n+1)/2}.
\]
Consequently,
\[
        \Ebb T_{n,p}
        \sim (q;q)_\infty\frac{n^n}{n!}q^{-n(n+1)/2},
\]
and
\[
        \Var(T_{n,p})
        \sim (q;q)_\infty^2\left(\frac{n^n}{n!}\right)^2q^{-n(n+1)}.
\]
\end{corollary}

\begin{proof}
Use the same block scales as in Lemma~\ref{lem:careless-mixing}: choose
\[
        b_n=n^3,
        \qquad h_n=\ceil{C_0n^2},
\]
with $C_0>1$, and put $L_n=b_n+h_n$.  By Lemma~\ref{lem:careless-mixing}, the set-valued chain is within $o(b_n\mu_n)$ of stationarity after $h_n$ steps.  Under the set-valued stationary law, the count coordinate has distribution $\nu_n$, and the event of a new entry into $[n]$ depends only on the count coordinate.  Hence Lemma~\ref{lem:careless-clump} gives stationary active-block probability
\[
        b_n\mu_n(1+o(1)).
\]
By the path-event total-variation comparison from Section~\ref{sec:stationary-entry-module}, the probability of this active-block event after the initial gap differs from its stationary value by at most the total-variation distance after the gap.  Hits during the initial gap only help.  Therefore, uniformly over all starting sets $S\subseteq[n]$,
\[
        \Pbb_S(T_{n,p}\le L_n)\ge \frac12 b_n\mu_n
\]
for all sufficiently large $n$.

By the strong Markov property at times $L_n,2L_n,\ldots$, for every integer $m\ge0$,
\[
        \Pbb_S(T_{n,p}>mL_n)
        \le \left(1-\frac12 b_n\mu_n\right)^m
        \le \exp\{-m b_n\mu_n/2\}.
\]
Since $h_n=o(b_n)$, we have $b_n/L_n\to1$.  Hence there are constants $C,c>0$ such that, for all large $n$ and all $y\ge0$,
\[
        \Pbb(\mu_nT_{n,p}>y)\le C e^{-cy}.
\]
This gives uniform integrability of $Y_n^r$ for every fixed $r\ge1$.  Combining this with Theorem~\ref{thm:careless-final}, which gives $Y_n\Rightarrow Z\sim\Exp(1)$, yields moment convergence.  The formulas for the mean and variance follow from the cases $r=1$ and $r=2$.
\end{proof}

\subsection{Scale comparison and marginal-independence heuristic}

\begin{corollary}[Careless logarithmic scale]\label{cor:careless-log-scale}
For fixed $p\in(0,1)$ and $q=1-p$,
\[
        \log T_{n,p}
        =\frac{n(n+1)}{2}\log\frac1q
          +\log\frac{n^n}{n!}
          +\OO_{\mathbb P}(1).
\]
Equivalently, by Stirling's formula,
\[
        \log T_{n,p}
        =\frac{n(n+1)}{2}\log\frac1{1-p}
          +n-\frac12\log(2\pi n)+\OO_{\mathbb P}(1).
\]
\end{corollary}

\begin{proof}
Theorem~\ref{thm:careless-final} gives $\mu_nT_{n,p}\Rightarrow\Exp(1)$.  Since $\Exp(1)$ is supported on $(0,\infty)$, this implies
\[
        \log(\mu_nT_{n,p})=\OO_{\mathbb P}(1).
\]
Theorem~\ref{thm:careless-sharp-constant} gives $\mu_n\sim(q;q)_\infty^{-1}R_n$, and therefore
\[
        \log T_{n,p}=\log R_n^{-1}+\OO_{\mathbb P}(1).
\]
Since
\[
        R_n=\frac{n!}{n^n}q^{n(n+1)/2},
\]
the first displayed formula follows.  Stirling's formula gives the second.
\end{proof}

\begin{remark}[Comparison with Cruciani--Dudeja and the marginal-independence scale]
Cruciani and Dudeja~\cite[Sections~4.2--4.3]{CrucianiDudeja} study the same post-loss hitting time and identify a substantial gap between their rigorous lower and upper bounds.  With
\[
        q_*:=\frac{1-p}{1-p+np},
\]
their lower bound and spatial-temporal mean-field analysis predict the scale $q_*^{-n}$, and they state that closing the quantitative gap should require a more refined Markov-chain analysis.

The high-tail analysis above supplies such a Markov-chain analysis in the fixed-$p$ post-loss regime.  It shows that the marginal mean-field scale is not the true stationary-entry scale.  Instead,
\[
        \Ebb T_{n,p}
        \sim
        (q;q)_\infty\,\frac{n^n}{n!}\,q^{-n(n+1)/2},
        \qquad q=1-p.
\]
Thus
\[
        \log \Ebb T_{n,p}
        =
        \frac{n(n+1)}{2}|\log q|+n+O(\log n).
\]
For fixed $p$, this is quadratic-exponential in $n$, whereas
\[
        q_*^{-n}=\left(\frac{np+q}{q}\right)^n
\]
has logarithm $n\log n+O(n)$.  The extra factor is therefore not a constant-order correction; it changes the logarithmic scale.  Its source is the ordered lucky climb through the high tail of the count chain.

The same point can be phrased as a failure of the naive product-marginal heuristic.  Stationarity of the mean count gives
\[
        \Ebb_{\nu_n}K=\frac{q}{p+q/n},
\]
so a one-coordinate marginal calculation suggests occupancy probability
\[
        \pi_{1,n}=\frac{q}{np+q}=q_*.
\]
Treating coupon presences as independent would put the all-present mass near $\pi_{1,n}^n$ and predict the reciprocal scale $q_*^{-n}$.  The cut-flux computation instead shows that the high tail is not governed by the independent one-point product heuristic, and gives the lucky-climb scale
\[
        \frac{n!}{n^n}q^{n(n+1)/2}.
\]
\end{remark}

\section{A combined clumsy-careless model}

The combined model checks that the careless high-tail entry mechanism is stable when selected-type refresh is layered with global thinning.  Fix $\alpha\in[0,1)$ and $\beta\in(0,1)$, and put $Q:=1-\alpha$ and $S:=1-\beta$.  At each round, one coupon type is drawn uniformly.  The selected type is refreshed to present with probability $Q$ and to absent with probability $\alpha$.  After this selected-type refresh, every currently present coupon is independently retained with probability $S$.  Completion is checked after the global thinning step.

Let $K_t$ be the number of coupon types present after thinning.  If $K_t=k$, then
\begin{equation}\label{eq:combined-transition}
        K_{t+1}\sim
        \begin{cases}
        \Bin(k-1,S)+\mathrm{Bernoulli}(QS),&\text{with probability }k/n,\\[1mm]
        \Bin(k,S)+\mathrm{Bernoulli}(QS),&\text{with probability }(n-k)/n.
        \end{cases}
\end{equation}
The first line is omitted when $k=0$.  The set-valued chain is irreducible and aperiodic: from any set all present coupons can be lost during thinning, from the empty set any prescribed set can be built by successive successful refreshes with no thinning losses, and $P(\emptyset,\emptyset)>0$.  By symmetry, the count-chain stationary law is the image of the set-valued stationary law under $S\mapsto |S|$.  Denote it by $\nu_n^{\alpha,\beta}$, and let
\[
        \mu_n^{\alpha,\beta}
        :=\Pbb_{\nu_n^{\alpha,\beta}}(K_0<n,K_1=n)
\]
be the stationary-entry flux into the post-thinning completion state.

For $0\le k<n$, the only way to move from $k$ to $k+1$ is to draw a missing coupon, refresh it successfully, and retain all $k+1$ present coupons through global thinning.  Thus
\[
        u_k^{\alpha,\beta}
        =\frac{n-k}{n}Q S^{k+1},
\]
and the associated lucky-climb weights are
\begin{equation}\label{eq:combined-lucky-scale}
        w_{n,k}^{\alpha,\beta}
        :=\prod_{r=0}^{k-1}u_r^{\alpha,\beta}
        =\frac{(n)_k}{n^k}Q^kS^{k(k+1)/2},
        \qquad w_{n,0}^{\alpha,\beta}:=1.
\end{equation}

\begin{theorem}[Combined clumsy-careless collector]\label{thm:combined-main}
For fixed $\alpha\in[0,1)$ and $\beta\in(0,1)$,
\[
        \mu_n^{\alpha,\beta}
        \sim
        \frac{1}{(S;S)_\infty}
        \frac{n!}{n^n}Q^nS^{n(n+1)/2}.
\]
If $T_n^{\alpha,\beta}$ is the first post-thinning completion time from the empty collection, then
\[
        \mu_n^{\alpha,\beta}T_n^{\alpha,\beta}
        \Rightarrow \Exp(1).
\]
Moreover, for every fixed $r\ge1$,
\[
        \Ebb\bigl(\mu_n^{\alpha,\beta}T_n^{\alpha,\beta}\bigr)^r\to r!.
\]
Consequently,
\[
        \Ebb T_n^{\alpha,\beta}
        \sim
        (S;S)_\infty
        \frac{n^n}{n!}Q^{-n}S^{-n(n+1)/2},
\]
and
\[
        \Var(T_n^{\alpha,\beta})
        \sim
        (S;S)_\infty^2
        \left(\frac{n^n}{n!}\right)^2
        Q^{-2n}S^{-n(n+1)}.
\]
\end{theorem}

\begin{proof}
The high-tail estimate in Proposition~\ref{prop:combined-high-tail-appendix} gives
\[
        \nu_n^{\alpha,\beta}(n)
        \sim
        \frac{1}{(S;S)_\infty}
        \frac{n!}{n^n}Q^nS^{n(n+1)/2}.
\]
From state $n$, the chain remains at $n$ for one more step only if the selected type is refreshed to present and all $n$ coupons survive thinning, so $P(n,n)=QS^n$.  Therefore
\[
        \mu_n^{\alpha,\beta}
        =\nu_n^{\alpha,\beta}(n)(1-QS^n)
        \sim \nu_n^{\alpha,\beta}(n),
\]
which proves the flux asymptotic.

Lemma~\ref{lem:combined-clump} gives the stationary one-block law, and Lemma~\ref{lem:combined-mixing} gives mixing and burn-in from the empty initial state.  The target-entry event for the set-valued chain is exactly the count event $\{K_{t-1}<n, K_t=n\}$, so the count-chain one-block law is the corresponding one-block law for the set-valued all-present target.  Theorem~\ref{thm:stationary-entry} therefore yields the exponential hitting limit.

Proposition~\ref{prop:combined-uniform-tails} gives uniform exponential tails for $\mu_n^{\alpha,\beta}T_n^{\alpha,\beta}$, hence uniform integrability of every fixed power.  Thus
\[
        \Ebb\bigl(\mu_n^{\alpha,\beta}T_n^{\alpha,\beta}\bigr)^r\to r!
        \qquad(r\ge1\text{ fixed}).
\]
The mean and variance asymptotics are the cases $r=1,2$, together with the flux asymptotic above.
\end{proof}

\begin{remark}[A singular boundary at zero global loss]
The assumption $\beta>0$ is essential.  If $\beta=0$ and $\alpha\in(0,1)$, then $S=1$, global thinning disappears, and the product $(S;S)_\infty$ vanishes.  This is not a removable limiting singularity: the combined model reduces to the pure clumsy collector with selected-type loss probability $\alpha$, whose completion scale is $(\alpha(1-\alpha)^n)^{-1}$ rather than the lucky-climb scale in Theorem~\ref{thm:combined-main}.  If $\alpha=0$ as well, the model degenerates further to the ordinary monotone coupon collector.
\end{remark}

\clearpage
\section{Conclusion}

The examples above show two ways non-monotone coupon collectors reach completion.  For the reset-button collector, the underlying model and beta-function expectation are due to Jockovi\'c and Todi\'c; our contribution is to recast the process through exact regeneration at the empty state, where the model-specific input is the successful-excursion probability.  For the clumsy and careless collectors, completion is a rare entry into a metastable completion state, and the model-specific input is the stationary-entry flux plus one-block clump control.  Together, these examples fit into the same terminal-condition viewpoint.

For the clumsy collector, the product Bernoulli stationary law makes the entry rate explicit and gives a fixed-$p$ exponential limit.  For the post-loss careless collector of Cruciani and Dudeja, the stationary high tail is not governed by the independent one-point product heuristic.  Stationary cut identities instead yield the stationary-entry flux scale
\[
        \frac{1}{(q;q)_\infty}\frac{n!}{n^n}q^{n(n+1)/2},
        \qquad q=1-p.
\]
This replaces the one-point marginal mean-field scale in the fixed-$p$ regime and changes the logarithmic completion scale from $n\log n+O(n)$ to $\frac12|\log q|n^2+O(n)$.  The same high-tail entry mechanism also governs the combined clumsy-careless model.  Across these stationary-entry models, the block construction gives uniform exponential tails for the normalized hitting times and hence convergence of all fixed moments to the corresponding exponential moments.

\clearpage
\appendix
\section{Reset-button asymptotic regimes}\label{app:reset-asymptotic-regimes}

\subsection{Equal-probability rare-success regimes}

Throughout this section the standard coupons are equally likely, so Lemma~\ref{lem:reset-uniform-laplace} applies.

\begin{lemma}[Verification of the rare-success hypotheses]\label{lem:reset-equal-hypotheses}
Assume either
\begin{enumerate}[label=\textnormal{(\alph*)},leftmargin=2.4em]
\item $\rho_n\to\rho\in(0,1)$, or
\item $\rho_n\sim\lambda/n$ for some fixed $\lambda>0$.
\end{enumerate}
Then the hypotheses of Theorem~\ref{thm:reset-rare-success} hold.  In addition, the optional right-tilt condition \eqref{eq:reset-right-tilt-condition} holds for every fixed $a\in(0,1)$.
\end{lemma}

\begin{proof}
Let $\phi_n$ be the ordinary uniform collector PGF\@.  From Lemma~\ref{lem:reset-uniform-laplace},
\[
        \log\phi_n(x)=\sum_{j=1}^n
        \log\left(\frac{(j/n)x}{1-(1-j/n)x}\right).
\]
Therefore
\[
        \frac{x\phi_n'(x)}{\phi_n(x)}
        =\sum_{j=1}^n\frac{1}{1-(1-j/n)x}.
\]
At $x=q_n$ this gives
\[
        \frac{q_n\phi_n'(q_n)}{s_n}
        =\sum_{j=1}^n\frac{1}{\rho_n+jq_n/n}.
\]
If $\rho_n\to\rho\in(0,1)$, this sum is $O(n)$, while Lemma~\ref{lem:reset-uniform-laplace} and Stirling's formula give $s_n$ exponentially small.  Hence
\[
        \rho_nq_n\phi_n'(q_n)
        =\rho_ns_n\frac{q_n\phi_n'(q_n)}{s_n}
        \to0.
\]
Moreover, for fixed $a\in(0,1)$,
\[
        \log\frac{\phi_n(q_ne^{a\rho_ns_n})}{\phi_n(q_n)}
        \le a\rho_ns_ne^{a\rho_ns_n}
        \frac{q_ne^{a\rho_ns_n}\phi_n'(q_ne^{a\rho_ns_n})}{\phi_n(q_ne^{a\rho_ns_n})}.
\]
The logarithmic derivative at $q_ne^{a\rho_ns_n}$ is still $O(n)$, while $s_n$ is exponentially small.  Thus the logarithm tends to zero, proving the right-tilt condition.

If $\rho_n\sim\lambda/n$, put
\[
        a_n:=\frac{n\rho_n}{q_n}.
\]
Then $a_n\to\lambda$ and $q_n\to1$.  Hence
\[
        \sum_{j=1}^n\frac{1}{\rho_n+jq_n/n}
        \le Cn\sum_{j=1}^n\frac1{j+\lambda/2}
        =O(n\log n).
\]
By Lemma~\ref{lem:reset-uniform-laplace} and the uniform gamma-ratio asymptotic for $a_n$ in a compact subset of $(0,\infty)$,
\[
        s_n=\Gamma(1+a_n)n^{-a_n}(1+o(1))=n^{-\lambda+o(1)}.
\]
Therefore
\[
        \rho_ns_nO(n\log n)=n^{-\lambda+o(1)}\log n\to0,
\]
which proves $\rho_nq_n\phi_n'(q_n)\to0$.  The same bound gives the optional right-tilt condition.  Indeed, for fixed $a\in(0,1)$,
\[
        a\rho_ns_n=o(\rho_n),
\]
so, uniformly for $x$ between $q_n$ and $q_ne^{a\rho_ns_n}$,
\[
        1-(1-j/n)x \asymp \rho_n+j/n,
        \qquad 1\le j\le n.
\]
Hence $x\phi_n'(x)/\phi_n(x)=O(n\log n)$ throughout this interval, and integrating the logarithmic derivative over an interval of logarithmic length $a\rho_ns_n+o(\rho_ns_n)$ gives
\[
        \log\frac{\phi_n(q_ne^{a\rho_ns_n})}{\phi_n(q_n)}=o(1).
\]
\end{proof}

\begin{corollary}[Fixed reset probability]\label{cor:reset-fixed-rho}
Assume $\rho_n\equiv\rho\in(0,1)$ and put $q=1-\rho$.  Then
\[
        s_n
        \sim
        \sqrt{2\pi\rho n}\left(q\rho^{\rho/q}\right)^n.
\]
Consequently
\[
        \rho\sqrt{2\pi\rho n}\left(q\rho^{\rho/q}\right)^nT_n
        \Rightarrow \Exp(1),
\]
and for every fixed integer $r\ge1$,
\[
        \Ebb\left[\rho\sqrt{2\pi\rho n}\left(q\rho^{\rho/q}\right)^nT_n\right]^r
        \longrightarrow r!.
\]
In particular,
\[
        \Ebb T_n
        \sim
        \frac{1}{\rho\sqrt{2\pi\rho n}}
        \left(q\rho^{\rho/q}\right)^{-n}.
\]
\end{corollary}

\begin{proof}
By Lemma~\ref{lem:reset-uniform-laplace}, with $a_n=n\rho/q$,
\[
        s_n=\frac{\Gamma(n+1)\Gamma(1+a_n)}{\Gamma(n+1+a_n)}.
\]
Stirling's formula gives
\[
        s_n\sim
        \sqrt{2\pi n\frac{\rho/q}{1+\rho/q}}
        \left(\frac{(\rho/q)^{\rho/q}}{(1+\rho/q)^{1/q}}\right)^n.
\]
Since $(\rho/q)/(1+\rho/q)=\rho$ and $1+\rho/q=1/q$, this simplifies to
\[
        s_n\sim\sqrt{2\pi\rho n}\left(q\rho^{\rho/q}\right)^n.
\]
Because $0<q\rho^{\rho/q}<1$, this also shows that $s_n\to0$.
Lemma~\ref{lem:reset-equal-hypotheses} verifies the hypotheses of
Theorem~\ref{thm:reset-rare-success} in the fixed-$\rho$ regime.  Hence
\[
        \rho s_nT_n\Rightarrow\Exp(1),
        \qquad
        \Ebb(\rho s_nT_n)^r\to r!
\]
for every fixed integer $r\ge1$.  Replacing $s_n$ by the asymptotic equivalent
above gives the displayed normalization and moment convergence.  Finally,
Corollary~\ref{cor:reset-exact-mean} gives
\[
        \Ebb T_n=\frac{1-s_n}{\rho s_n}\sim\frac1{\rho s_n},
\]
which is the stated expectation asymptotic.
\end{proof}

\begin{corollary}[Polynomial reset regime]\label{cor:reset-lambda-over-n}
Assume
\[
        \rho_n\sim\frac{\lambda}{n},
        \qquad \lambda>0,
\]
and put
\[
        a_n:=\frac{n\rho_n}{1-\rho_n}.
\]
Then
\[
        s_n=\Gamma(1+a_n)n^{-a_n}(1+o(1)),
\]
and therefore
\[
        \rho_n\Gamma(1+a_n)n^{-a_n}T_n
        \Rightarrow\Exp(1).
\]
Moreover, all fixed positive integer moments converge to the corresponding exponential moments under this normalization.

If, in addition,
\[
        (a_n-\lambda)\log n\to0,
\]
then the simpler normalization is valid:
\[
        \frac{\lambda\Gamma(\lambda+1)}{n^{\lambda+1}}T_n
        \Rightarrow\Exp(1),
\]
again with convergence of every fixed positive integer moment.
\end{corollary}

\begin{proof}
By Lemma~\ref{lem:reset-uniform-laplace},
\[
        s_n=\frac{\Gamma(n+1)\Gamma(1+a_n)}{\Gamma(n+1+a_n)}.
\]
Since $a_n\to\lambda\in(0,\infty)$, the gamma-ratio asymptotic, uniformly for $a_n$ in any compact subinterval of $(0,\infty)$, gives
\[
        \frac{\Gamma(n+1)}{\Gamma(n+1+a_n)}=n^{-a_n}(1+o(1)).
\]
Thus
\[
        s_n=\Gamma(1+a_n)n^{-a_n}(1+o(1)).
\]
The first distributional and moment conclusions follow from Theorem~\ref{thm:reset-rare-success} and Lemma~\ref{lem:reset-equal-hypotheses}, because
\[
        \frac{\rho_n\Gamma(1+a_n)n^{-a_n}}{\rho_ns_n}\to1.
\]
If $(a_n-\lambda)\log n\to0$, then $n^{-a_n}\sim n^{-\lambda}$, $\Gamma(1+a_n)\to\Gamma(1+\lambda)$, and $\rho_n\sim\lambda/n$.  Therefore
\[
        \rho_n\Gamma(1+a_n)n^{-a_n}
        \sim \frac{\lambda\Gamma(\lambda+1)}{n^{\lambda+1}},
\]
which gives the simplified normalization.
\end{proof}

\begin{corollary}[Reset coupon equally likely as each standard coupon]\label{cor:reset-equal-reset-standard}
Assume
\[
        \rho_n=\frac1{n+1},
        \qquad
        p_{i,n}=\frac1{n+1},
        \quad 1\le i\le n.
\]
Then
\[
        s_n=\frac1{n+1},
\]
\[
        \Ebb T_n=n(n+1),
\]
and
\[
        \frac{T_n}{n(n+1)}\Rightarrow\Exp(1).
\]
All fixed positive integer moments converge to those of $\Exp(1)$.
\end{corollary}

\begin{proof}
In this case $q_n=n/(n+1)$ and
\[
        a_n=\frac{n\rho_n}{q_n}=1.
\]
Therefore Lemma~\ref{lem:reset-uniform-laplace} gives
\[
        s_n=\frac{\Gamma(n+1)\Gamma(2)}{\Gamma(n+2)}=\frac1{n+1}.
\]
Corollary~\ref{cor:reset-exact-mean} gives
\[
        \Ebb T_n=\frac{1-s_n}{\rho_ns_n}
        =\frac{1-1/(n+1)}{(1/(n+1))(1/(n+1))}
        =n(n+1).
\]
Moreover,
\[
        \rho_ns_n=(n+1)^{-2}.
\]
Theorem~\ref{thm:reset-rare-success} therefore gives
\[
        \frac{T_n}{(n+1)^2}\Rightarrow\Exp(1)
\]
with convergence of every fixed positive integer moment.  Since $n(n+1)\sim (n+1)^2$, the displayed normalization follows.
\end{proof}

\subsection{The negligible-reset regime}

The rare-success theorem applies when one successful reset-free excursion is itself rare.  If resets are much rarer than the ordinary coupon collector completion window, the ordinary Gumbel law returns.

\begin{theorem}[Negligible resets give the ordinary Gumbel law]\label{thm:reset-negligible-reset}
Assume equal standard probabilities and
\[
        \rho_n n\log n\to0.
\]
Let $T_n$ be the reset-button completion time.  Then, for every $y\in\Rbb$,
\[
        \Pbb\left(T_n\le n\log n+ny\right)
        \longrightarrow
        \exp(-e^{-y}).
\]
Equivalently,
\[
        \frac{T_n-n\log n}{n}\Rightarrow G,
\]
where $G$ is the standard coupon-collector Gumbel variable with distribution function
\[
        \Pbb(G\le y)=\exp(-e^{-y}).
\]
\end{theorem}

\begin{proof}
Couple the reset-button process to an ordinary uniform coupon collector by using the same sequence of standard coupon labels whenever no reset occurs.  Let
\[
        t_n(y):=\floor{n\log n+ny}.
\]
The probability that at least one reset occurs before time $t_n(y)$ is at most
\[
        \rho_n t_n(y)=o(1).
\]
On the event that no reset occurs before $t_n(y)$, the reset-button process and the ordinary uniform coupon collector have the same completion status at time $t_n(y)$.  Hence
\[
        \left|
        \Pbb(T_n\le t_n(y))
        -\Pbb(C_n^{\mathrm{ord}}\le t_n(y))
        \right|
        \le \Pbb(\text{at least one reset before }t_n(y))=o(1),
\]
where $C_n^{\mathrm{ord}}$ is the ordinary uniform coupon collector completion time in calendar draws.  The classical coupon collector theorem \cite{ErdosRenyi,Feller} gives
\[
        \Pbb(C_n^{\mathrm{ord}}\le n\log n+ny)\to\exp(-e^{-y}).
\]
This proves the claim.
\end{proof}

\begin{remark}[Transition window]
The critical transition between Theorem~\ref{thm:reset-negligible-reset} and the rare-success regime occurs around
\[
        \rho_n\asymp \frac1{n\log n}.
\]
This critical window requires a separate analysis.  In that window, the identity $s_n=\Ebb q_n^{C_n}$ and the classical estimate $C_n/(n\log n)\to1$ indicate that, if $\rho_nn\log n\to\lambda\in(0,\infty)$, one reset-free attempt has success probability tending to $e^{-\lambda}$, not to zero.  The regenerative representation points to a limiting description on the $n\log n$ scale in terms of a finite geometric sum of reset-cycle lengths, with success probability related to $e^{-\lambda}$.  A full treatment lies outside the present scope.
\end{remark}

\section{Combined stationary-entry details}\label{app:combined-high-tail}

This appendix contains the details deferred from the combined clumsy-careless model: the high-tail comparison, one-block clump control, mixing and burn-in, and the uniform tail bound used for moment convergence.  The high-tail proof repeats the careless argument with the global survival parameter $q$ replaced by $S$ and with the selected-refresh factor $Q^k$ carried through the lucky-climb weights.

The proof below is the combined analogue of Theorem~\ref{thm:careless-high-tail} and Lemmas~\ref{lem:infinite-chain}--\ref{lem:normalized-propagation}.  The only new feature is the selected-refresh factor $Q$ carried through the lucky-climb weights.

\begin{proposition}[Combined high-tail comparison and constant]\label{prop:combined-high-tail-appendix}
Fix $\alpha\in[0,1)$ and $\beta\in(0,1)$, and put $Q=1-\alpha$ and $S=1-\beta$.  Let $\nu_n^{\alpha,\beta}$ be the stationary law of the combined count chain \eqref{eq:combined-transition}, and let $w_{n,k}^{\alpha,\beta}$ be defined by \eqref{eq:combined-lucky-scale}.  There are constants $0<c<C<\infty$ and an integer $M<\infty$, depending only on $\alpha,\beta$, such that, for all sufficiently large $n$ and all $M\le k\le n$,
\[
        c w_{n,k}^{\alpha,\beta}
        \le \nu_n^{\alpha,\beta}(k)
        \le C w_{n,k}^{\alpha,\beta}.
\]
Moreover,
\[
        \nu_n^{\alpha,\beta}(n)
        \sim
        \frac{1}{(S;S)_\infty}
        \frac{n!}{n^n}Q^nS^{n(n+1)/2}.
\]
\end{proposition}

\begin{proof}
For $0\le k<n$, the upward transition probability is
\[
        u_k^{\alpha,\beta}=P(k,k+1)=\frac{n-k}{n}QS^{k+1},
\]
and $w_{n,k+1}^{\alpha,\beta}=w_{n,k}^{\alpha,\beta}u_k^{\alpha,\beta}$.

We first obtain the uniform scale comparison.  Under stationarity,
\[
        \Ebb_{\nu_n^{\alpha,\beta}}K
        =\frac{SQ}{1-S(1-1/n)}
        \le \frac{SQ}{1-S}.
\]
Choose a fixed integer $M_0$ larger than $4SQ/(1-S)$.  Markov's inequality gives
\[
        \nu_n^{\alpha,\beta}(\{0,1,\ldots,M_0\})\ge\frac34.
\]
From any state $j\le M_0$, after the selected-type refresh there are at most $M_0+1$ present coupons.  All of them are lost during thinning with probability at least $(1-S)^{M_0+1}$.  Hence
\[
        \nu_n^{\alpha,\beta}(0)
        \ge \frac34(1-S)^{M_0+1}=:c_0>0.
\]
Starting from $0$, the strictly increasing path $0\to1\to\cdots\to k$ has probability
\[
        \prod_{r=0}^{k-1}u_r^{\alpha,\beta}=w_{n,k}^{\alpha,\beta}.
\]
Therefore
\[
        \nu_n^{\alpha,\beta}(k)
        \ge c_0 w_{n,k}^{\alpha,\beta},
        \qquad 0\le k\le n.
\]

For the upper bound, put $A_k=\{0,1,\ldots,k\}$.  Upward jumps have size at most one, so stationarity gives the cut-flux identity
\[
        \nu_n^{\alpha,\beta}(k)u_k^{\alpha,\beta}
        =\sum_{j=k+1}^n \nu_n^{\alpha,\beta}(j)P(j,A_k).
\]
Keeping only the $j=k+1$ term,
\[
        \nu_n^{\alpha,\beta}(k+1)P(k+1,A_k)
        \le \nu_n^{\alpha,\beta}(k)u_k^{\alpha,\beta}.
\]
Let $\ell=k+1$.  Starting from level $\ell$, ending at level at least $\ell$ after one step requires a very small number of thinning losses.  If the selected coupon is already present, then an upper bound is obtained by requiring no loss among the $\ell$ currently present coupons.  If the selected coupon is missing, then either the selected refresh fails and again no loss among the original $\ell$ coupons is needed, or the refresh succeeds and at most one loss among at most $\ell+1$ present coupons is allowed.  Therefore
\[
\begin{aligned}
        P(\ell,\{\ell,\ell+1,\ldots,n\})
        &\le S^\ell+Q\left(S^{\ell+1}+(\ell+1)(1-S)S^\ell\right)  \\
        &\le C_{\alpha,\beta}\ell S^\ell,
\end{aligned}
\]
and therefore
\[
        P(k+1,A_k)\ge 1-C_{\alpha,\beta}(k+1)S^{k+1}.
\]
Choose $M\ge M_0$ so large that the last error is at most $1/2$ for $k\ge M$.  Iterating gives
\[
        \nu_n^{\alpha,\beta}(k)
        \le \nu_n^{\alpha,\beta}(M)
        \prod_{r=M}^{k-1}u_r^{\alpha,\beta}
        \prod_{r=M}^{k-1}\frac{1}{1-C_{\alpha,\beta}(r+1)S^{r+1}}.
\]
The infinite product is finite because $\sum rS^r<\infty$.  Also $\nu_n^{\alpha,\beta}(M)\le1$, while $\prod_{r=0}^{M-1}u_r^{\alpha,\beta}$ is bounded below by a positive constant depending only on $\alpha,\beta$ for all sufficiently large $n$.  Hence
\[
        \nu_n^{\alpha,\beta}(k)
        \le C w_{n,k}^{\alpha,\beta},
        \qquad M\le k\le n.
\]
This proves the uniform comparison.

We now identify the top constant.  The local limiting chain is obtained by fixing $k$ and sending $n\to\infty$ in \eqref{eq:combined-transition}:
\[
        Y_{t+1}\mid Y_t=k\sim \Bin(k,S)+\mathrm{Bernoulli}(QS).
\]
Its stationary law is
\[
        Y_\infty\stackrel d=\sum_{r=1}^\infty B_r,
        \qquad B_r\sim\mathrm{Bernoulli}(QS^r),
\]
with the $B_r$ independent.  Old particles with parameter $QS^r$ survive one more thinning step and become particles with parameter $QS^{r+1}$, while the refreshed selected coupon contributes a new Bernoulli$(QS)$ particle.  This law is a stationary probability distribution by direct substitution.  Since the limiting chain is irreducible, positive recurrence follows from \cite[Theorem~21.13]{LevinPeres}; uniqueness of the stationary distribution then follows from irreducibility.

Let $\varpi$ denote this limiting stationary law.  Its tail satisfies
\begin{equation}\label{eq:appendix-combined-infinite-tail}
        \varpi(k)=\Pbb(Y_\infty=k)
        \sim \frac{Q^kS^{k(k+1)/2}}{(S;S)_\infty}.
\end{equation}
To prove this, write the generating function as
\[
        G_Q(z)=\prod_{r=1}^{\infty}(1-QS^r+QS^rz)
              =\prod_{r=1}^{\infty}(1+Q(z-1)S^r).
\]
By the $q$-binomial theorem \cite[Chapter~2]{AndrewsPartitions}, with base $S$,
\[
        G_Q(z)=\sum_{m=0}^{\infty}
        \frac{Q^mS^{m(m+1)/2}}{(S;S)_m}(z-1)^m.
\]
Extracting the coefficient of $z^k$ gives
\[
        \varpi(k)=\sum_{m=k}^{\infty}(-1)^{m-k}\binom{m}{k}
        \frac{Q^mS^{m(m+1)/2}}{(S;S)_m}.
\]
With $m=k+\ell$,
\[
\frac{\varpi(k)}{Q^kS^{k(k+1)/2}}
        =\sum_{\ell=0}^{\infty}(-1)^\ell\binom{k+\ell}{\ell}
        \frac{Q^\ell S^{\ell k+\ell(\ell+1)/2}}{(S;S)_{k+\ell}}.
\]
The $\ell=0$ term tends to $(S;S)_\infty^{-1}$.  Since $0<Q\le1$, the absolute value of the remaining terms is at most
\[
        \frac1{(S;S)_\infty}
        \sum_{\ell=1}^{\infty}\binom{k+\ell}{\ell}S^{\ell k}
        =\frac{(1-S^k)^{-k-1}-1}{(S;S)_\infty},
\]
which tends to zero because $(k+1)S^k\to0$.  This proves \eqref{eq:appendix-combined-infinite-tail}.

We next prove local convergence of the finite stationary laws to $\varpi$.  Tightness follows from the stationary mean bound above.  For fixed $k,j$ the finite transition probabilities converge to the limiting transition probabilities.  Let a subsequence of $\nu_n^{\alpha,\beta}$ converge weakly to a law $\nu$.  For fixed $j$, stationarity gives
\[
        \nu_n^{\alpha,\beta}(j)=\sum_{k=0}^n\nu_n^{\alpha,\beta}(k)P_n^{\alpha,\beta}(k,j).
\]
We justify passage to the limit.  Since ending at a fixed level $j$ from a large level $k$ requires losing all but at most $j$ of the old coupons, there is a constant $C_{j,\alpha,\beta}<\infty$ such that, for all large $k$ and uniformly in $n$,
\[
        P_n^{\alpha,\beta}(k,j)
        \le C_{j,\alpha,\beta}(k+1)^j(1-S)^{k-j}.
\]
Consequently,
\[
        \sup_n\sum_{k>K}\nu_n^{\alpha,\beta}(k)P_n^{\alpha,\beta}(k,j)\to0,
        \qquad K\to\infty,
\]
and the same estimate holds for the limiting kernel.  Passing first to the limit on finite partial sums and then sending the truncation to infinity shows that $\nu$ is stationary for the limiting chain.  By uniqueness of the limiting stationary law, $\nu=\varpi$.  Hence the full sequence converges locally to $\varpi$.

As in Lemma~\ref{lem:careless-local-convergence}, local convergence and positivity of $\varpi$ on finite sets allow a diagonal choice of a deterministic sequence $m_n\to\infty$ with $m_n=o(\sqrt n)$ such that
\[
        \frac{\nu_n^{\alpha,\beta}(m_n)}{\varpi(m_n)}\to1.
\]
Since $(n)_{m_n}/n^{m_n}\to1$, \eqref{eq:appendix-combined-infinite-tail} gives
\[
        a_{n,m_n}^{\alpha,\beta}
        :=\frac{\nu_n^{\alpha,\beta}(m_n)}{w_{n,m_n}^{\alpha,\beta}}
        \longrightarrow \frac{1}{(S;S)_\infty}.
\]

It remains to propagate this normalized constant from $m_n$ to $n$.  For $k\ge M$, the same cut-flux identity may be written as
\[
        \nu_n^{\alpha,\beta}(k)u_k^{\alpha,\beta}
        =\nu_n^{\alpha,\beta}(k+1)d_k+R_k,
\]
where
\[
        d_k:=P(k+1,A_k),
        \qquad
        R_k:=\sum_{j=k+2}^n\nu_n^{\alpha,\beta}(j)P(j,A_k).
\]
As above,
\[
        1-d_k\le C_{\alpha,\beta}kS^k.
\]
Using the uniform upper comparison and $P(j,A_k)\le1$,
\[
        R_k\le C\sum_{j=k+2}^n w_{n,j}^{\alpha,\beta}.
\]
For $j\ge k+2$,
\[
        \frac{w_{n,j}^{\alpha,\beta}}{w_{n,k+1}^{\alpha,\beta}}
        =\prod_{r=k+1}^{j-1}u_r^{\alpha,\beta}
        \le S^{(j-k-1)(k+2)},
\]
and therefore
\[
        \frac{R_k}{w_{n,k+1}^{\alpha,\beta}}
        \le C\sum_{h=1}^{\infty}S^{h(k+2)}
        =C\frac{S^{k+2}}{1-S^{k+2}}
        =\OO_{\alpha,\beta}(S^k).
\]
The uniform lower comparison gives $a_{n,k}^{\alpha,\beta}\ge c>0$.  Hence
\[
        \rho_{n,k}:=\frac{R_k}{a_{n,k}^{\alpha,\beta}w_{n,k+1}^{\alpha,\beta}}
        =\OO_{\alpha,\beta}(S^k).
\]
For all large $k$, both $\rho_{n,k}$ and $1-d_k$ are less than $1/2$, and the identity
\[
        a_{n,k}^{\alpha,\beta}
        =d_k a_{n,k+1}^{\alpha,\beta}
          +\frac{R_k}{w_{n,k+1}^{\alpha,\beta}}
\]
implies
\[
        \frac{a_{n,k+1}^{\alpha,\beta}}{a_{n,k}^{\alpha,\beta}}
        =\frac{1-\rho_{n,k}}{d_k}
        =1+\OO_{\alpha,\beta}(kS^k).
\]
Since $\sum_{k\ge m_n}kS^k=o(1)$, summing logarithms gives, uniformly in the upper endpoint $k\le n$,
\[
        a_{n,k}^{\alpha,\beta}
        =a_{n,m_n}^{\alpha,\beta}(1+o(1)).
\]
Taking $k=n$ gives
\[
        \nu_n^{\alpha,\beta}(n)
        =a_{n,n}^{\alpha,\beta}w_{n,n}^{\alpha,\beta}
        \sim
        \frac{1}{(S;S)_\infty}
        \frac{n!}{n^n}Q^nS^{n(n+1)/2},
\]
which is the desired top-tail constant.
\end{proof}

\subsection{Clump control, mixing, and burn-in}

\begin{lemma}[Combined one-block clump control]\label{lem:combined-clump}
Let $b_n\le n^M$ for some fixed $M<\infty$.  With
\[
        E_t:=\1\{K_{t-1}<n,\ K_t=n\},
        \qquad
        N_b:=\sum_{t=1}^{b_n}E_t,
\]
one has
\[
        \Pbb_{\nu_n^{\alpha,\beta}}(N_b\ge1)
        =b_n\mu_n^{\alpha,\beta}(1+o(1)).
\]
\end{lemma}

\begin{proof}
The first moment is exact by stationarity:
\[
        \Ebb_{\nu_n^{\alpha,\beta}}N_b=b_n\mu_n^{\alpha,\beta}.
\]
For any starting state $x$ and any $u\ge1$,
\[
        \Pbb_x(K_u=n)\le QS^n\le S^n,
\]
because the last transition into state $n$ requires all $n$ present coupons to survive the thinning step, and also requires the selected-type refresh to succeed when the selected type is relevant.  Since $E_u\le\1\{K_u=n\}$,
\[
        \Ebb_n\sum_{u=1}^{b_n}E_u\le b_nS^n=o(1).
\]
Using the Markov property after a first entry,
\[
        \Ebb_{\nu_n^{\alpha,\beta}}[(N_b)_2]
        \le 2b_n\mu_n^{\alpha,\beta}\, b_nS^n
        =o(b_n\mu_n^{\alpha,\beta}).
\]
The inequalities
\[
        0\le \Ebb_{\nu_n^{\alpha,\beta}} N_b
        -\Pbb_{\nu_n^{\alpha,\beta}}(N_b\ge1)
        \le \Ebb_{\nu_n^{\alpha,\beta}}[(N_b)_2]
\]
complete the proof.
\end{proof}

\begin{lemma}[Combined fast mixing and negligible burn-in]\label{lem:combined-mixing}
Let $\widehat\pi_n^{\alpha,\beta}$ be the stationary law of the set-valued combined chain.  Then
\[
        \sup_{A\subseteq[n]}
        \|P^t(A,\cdot)-\widehat\pi_n^{\alpha,\beta}\|_{\TV}
        \le nS^t.
\]
Moreover, from the empty initial collection, with
\[
        b_n=n^3,
        \qquad
        h_n=r_n=\ceil{Cn^2},
\]
for any fixed $C>1$, all mixing and burn-in assumptions in Theorem~\ref{thm:stationary-entry} hold for the post-thinning completion state.
\end{lemma}

\begin{proof}
Couple two set-valued chains using the same selected coupon, the same selected-type refresh bit, and the same thinning bits.  A discrepancy at a fixed coupon type is removed if that type is selected.  If it is not selected, the discrepancy can persist for one more round only if the present copy survives the global thinning, an event of probability $S$.  Thus a fixed initial discrepancy survives one round with probability at most $(1-1/n)S\le S$, and the union bound gives the displayed coupling estimate.

By Proposition~\ref{prop:combined-high-tail-appendix} and the identity \(\mu_n^{\alpha,\beta}=\nu_n^{\alpha,\beta}(n)(1-QS^n)\),
\[
        \log(b_n\mu_n^{\alpha,\beta})
        =-\frac{|\log S|}{2}n^2+\OO_{\alpha,\beta}(n)
\]
for $b_n=n^3$, whereas
\[
        \log(nS^{h_n})=\log n-C|\log S|n^2+\OO(1).
\]
Since $C>1$, $nS^{h_n}=o(b_n\mu_n^{\alpha,\beta})$.  Also $h_n=o(b_n)$ and $b_n\mu_n^{\alpha,\beta}\to0$.

For burn-in from the empty collection, completion after a fixed positive round requires all $n$ coupons to survive the thinning step and, if necessary, the selected coupon to be refreshed successfully.  Hence
\[
        \Pbb_\emptyset(K_t=n)\le QS^n\le S^n,
        \qquad t\ge1.
\]
The union bound gives
\[
        \Pbb_\emptyset(T_n^{\alpha,\beta}\le r_n)
        \le r_nS^n=o(1).
\]
The coupling bound gives mixing by time $r_n$, and $r_n\mu_n^{\alpha,\beta}=o(1)$ follows from the flux asymptotic.
\end{proof}

\subsection{Uniform tails for the combined collector}

\begin{proposition}[Combined uniform exponential tails]\label{prop:combined-uniform-tails}
For the fixed-parameter combined collector, there are constants $C,c\in(0,\infty)$, depending only on $\alpha,\beta$, such that
\[
        \sup_{A\subseteq[n]}
        \Pbb_A\left(\mu_n^{\alpha,\beta}T_n^{\alpha,\beta}>x\right)
        \le Ce^{-cx},
        \qquad x\ge0,
\]
for all sufficiently large $n$.  In particular, the bound holds for the empty initial collection.
\end{proposition}

\begin{proof}
Use
\[
        b_n=n^3,
        \qquad h_n=\ceil{C_0n^2},
        \qquad L_n=b_n+h_n,
\]
with $C_0>1$.  The set-valued mixing estimate in Lemma~\ref{lem:combined-mixing}, projected to the count coordinate, gives
\[
        \sup_x\|P^{h_n}(x,\cdot)-\nu_n^{\alpha,\beta}\|_{\TV}
        =o(b_n\mu_n^{\alpha,\beta}).
\]
By Lemma~\ref{lem:combined-clump}, a stationary active block of length $b_n$ contains a new entry with probability $b_n\mu_n^{\alpha,\beta}(1+o(1))$.  By the path-event total-variation comparison used in Section~\ref{sec:stationary-entry-module}, the probability of this active-block entry event differs from its stationary value by at most the total-variation distance after the initial gap.  Thus, uniformly over all starting sets $A\subseteq[n]$,
\[
        \Pbb_A(T_n^{\alpha,\beta}\le L_n)
        \ge \frac12 b_n\mu_n^{\alpha,\beta}
\]
for all sufficiently large $n$.  The strong Markov property at times $L_n,2L_n,\ldots$ gives, for every integer $m\ge0$,
\[
        \Pbb_A(T_n^{\alpha,\beta}>mL_n)
        \le \left(1-\frac12b_n\mu_n^{\alpha,\beta}\right)^m
        \le \exp\{-m b_n\mu_n^{\alpha,\beta}/2\}.
\]
Since $L_n/b_n\to1$, the displayed exponential tail bound follows.
\end{proof}

\end{document}